\documentclass[]{siamltex}
\usepackage{amssymb,amsmath}
\usepackage{color}
\usepackage{graphics}
\usepackage{graphicx}

\def\crr{\cr\noalign{\vskip2mm}}

\newcommand{\etal}{\textit{et al.~}}

\newtheorem{remark}{Remark}[section]

\newcommand{\curly}[1]{\left\{#1\right\}}

\usepackage{caption}
\usepackage{subcaption}

\def\m{\mu}

\def\a{\alpha}
\def\g{\gamma}

\def\l{\lambda}

\def\cI{{\cal I}}

\def\cR{{\cal R}}

\def\Re {\cR e }
\def\Im {\cI m }

\renewcommand{\theequation}{\arabic{section}.\arabic{equation}}

\usepackage[markup=underlined]{changes}
\definechangesauthor[color=blue]{ck}

\usepackage{color}
\usepackage{xcolor}

\definecolor{darkgreen}{RGB}{10,160,20}
\usepackage{enumitem}
\usepackage[hidelinks]{hyperref}
\usepackage{cleveref}

\title{\bf Stability and Optimal Decay Rates for Abstract Systems with Thermal Damping of Cattaneo's Type	\thanks{CHXD was supported by the China Scholarship Council, ZJH was supported by the National Natural Science Foundation of China (grants No. 62073236). QZ was supported by the National Natural Science Foundation of China (grants No. 12271035, 12131008) and Beijing Municipal Natural Science Foundation (grant No. 1232018).}}

\author{
		Chenxi Deng\thanks{School of Mathematics and Statistics,
			Beijing Institute of Technology, Beijing 100081, China (email: chenxideng@bit.edu.cn).}
		\and Zhong-Jie Han\thanks{School of Mathematics, BIIT Lab, Tianjin University,
				Tianjin 300354,   China (email: zjhan@tju.edu.cn)}
		\and Zhaobin Kuang\thanks{Computer Science Department, Stanford University, Stanford 94305, U.S.A.(email: zhaobin.kuang@gmail.com).}\and
		Qiong Zhang\thanks{Corresponding Author.  School of Mathematics and Statistics, Beijing Key Laboratory on MCAACI,
			Beijing Institute of Technology, Beijing 100081, China (email: zhangqiong@bit.edu.cn)}
	}

\begin{document}
		
\maketitle

\begin{abstract}
This paper studies the stability of an abstract thermoelastic system with Cattaneo's law, which describes finite heat propagation speed in a medium.  We introduce a region of parameters containing coupling, thermal dissipation, and possible inertial characteristics.
The region is partitioned into distinct subregions based on the spectral properties of the infinitesimal generator of the corresponding semigroup. By a careful estimation of the resolvent operator on the imaginary axis, we obtain distinct polynomial decay rates for systems with parameters located in different subregions. Furthermore, the optimality of these decay rates is proved. Finally, we apply our results to several coupled systems of partial differential equations.
\end{abstract}
	
	
	
	\pagestyle{myheadings}
	\thispagestyle{plain}
	\markboth{C. DENG, Z.J. HAN, Z. KUANG AND Q. ZHANG}{  Stabilization for  Elastic Systems with Thermal Damping of Cattaneo's Type}

	\section{Introduction}\label{sec:intro}
\setcounter{equation}{0}
In this paper, we consider an abstract thermoelastic system that describes the interaction between heat conduction and material deformation in solids. The model is defined as follows:
\begin{equation}\label{101}
\left\{
\begin{array}{l}
u_{tt} +mA^\gamma  u_{tt} +\sigma A u -A^\alpha\theta =0,  \\ \noalign{\medskip}  \displaystyle
\theta_t -A^{\frac{\beta}{2}}q +A^\alpha u_t =0,  \\ \noalign{\medskip}  \displaystyle
\tau q_t+q+A^{\frac{\beta}{2}}\theta=0,  \\ \noalign{\medskip}  \displaystyle
u(0)= u_0, \;\; u_t(0) = u_1,\;\; \theta(0) =\theta_0,\;\; q(0) = q_0,
\end{array} \right.
\end{equation}
where $A$ is a self-adjoint, positive definite operator with compact resolvent on the complex Hilbert space $H$.
The constants $m\geq 0, \; \sigma>0,\; \tau>0 $  are the inertial term,  wave speed, and the relaxation parameter,  respectively.  We define a region of parameters by $    E :=\{(\alpha,\;\beta,\;\gamma)\,|\,(\alpha,\;\beta,\;\gamma)  \in [0,1]\times[0,1]\times (0,1]\}$, where $\alpha,\beta,\gamma$ describe the coupling, thermal damping, and inertial characteristics.  
The case where $\gamma=0$ is omitted here because it can be included in the case where $m=0$.  Throughout this paper, we use $(\cdot,\,\cdot)$ and $\|\cdot\|$  to denote the inner product and norm in  $H$, respectively.

  If $\tau = 0$, the second and third equations in \eqref{101} reduce to an abstract system that covers the classical heat
equations subject to Fourier's law.   However, Fourier's law implies that all disturbances propagate at infinite speed, which is unacceptable in some physical processes such as high-frequency thermal phenomena and microscale heat conduction (\cite{Chand, Jou}).
Because of this shortcoming, various non-Fourier heat flux laws (e.g., Cattaneo's law) have been developed since the 1940s.
In Cattaneo's theory (\cite{Cattaneo,vernotte}),  a thermal relaxation parameter $\tau>0$ is introduced. This resolves the paradox of infinite speed of heat transfer in Fourier's law and characterizes the wave-like motion of heat, also referred to as the second sound in physics. Therefore, we say the system \eqref{101} follows  Fourier's law (or with thermal damping of Fourier's type) if $\tau=0$ and follows Cattaneo's law  (or with thermal damping of Cattaneo's type) if $\tau\neq0$, respectively.

On the other hand, the  natural energy of system \eqref{101} is defined by
\begin{equation*}\label{diss}
\mathcal{E}(t):=\frac{1}{2}\big(\|u_t\|^2+m\|A^{\frac{\gamma}{2}}u_t\|^2+\sigma\|A^{\frac{1}{2}} u\|^2+\|\theta\|^2+\tau\|q\|^2\big).
\end{equation*}
A direct computation gives
\begin{equation*}
{ \mathrm{d} \over \mathrm{d} t} \mathcal{E}(t)=-\|q\|^2 \leq 0,
\end{equation*}
which implies that the energy is decreasing on $[0,\infty)$.
Consequently, one natural question is to understand the asymptotic behavior (i.e., the stability when t tends to $\infty$) of the dissipative system \eqref{101}.

The stability of abstract thermoelastic systems can be traced back to the early 1990s. In 1993, Russell (\cite{russell}) introduced an abstract framework of an indirectly damped system, i.e., a conservative system coupled with a directly damped system. This is different from directly damped systems which are obtained by inserting dissipative terms directly into the originally conservative system. Furthermore, Russell claimed that, ``{\it at the present time, it does not appear possible to give a result for indirect damping mechanisms which even approaches the known direct damping results just listed in mathematical generality." }

 Inspired by Russell, research along this direction started, first with the system \eqref{101} when $\tau=0$, motivated by thermal elastic equations subject to Fourier's law. 
In 1996, Rivera and Racke (\cite{rivera1})  studied the regularity and exponential stability of the system \eqref{101} with $\tau=m=0$; They also provided specific examples of that system. Later,  Ammar-Khodja \emph{et al.} (\cite{Ammar}) identified the parameters regions for exponential stability of system \eqref{101} when $\tau=m=0$.
 Extending these results,  Hao and Liu (\cite{haoliu1}) obtained the optimal polynomial stability of the system beyond the regions of exponential stability presented in \cite{Ammar}. Further details on the case $\tau=m=0$ can be found in \cite{Avalos,conti,Kim,Lasiecka,LiuLiu} and the references therein.
Regarding the case where $\tau=0$ and $m\ne0$, Dell'Oro \emph{et al.} (\cite{rivera2}) investigated the stability of system \eqref{101} under the conditions  $\beta=\gamma=\frac{1}{2}$ and $0\leq\alpha\leq\frac{3}{4}$. To the best of our knowledge, this was the first result considering the inertial term.  Later, Fern\'{a}ndez Sare \emph{et al.} 
(\cite{liure}) provided a comprehensive overview of the exponential stability and polynomial stability for parameters {$(\alpha,\beta)\in[0,1]\times [0,1]$} and $\gamma\in(0,1]$. The optimality of the polynomial decay rates in \cite{liure} was proved in \cite{kuang}.
So far, the research on stability and optimal decay rates of the thermoelastic system \eqref{101} following Fourier's law has been relatively complete.

As pointed out earlier, studies of thermoelastic systems following Cattaneo's law were investigated to revise the downsides of Fourier's law. In 2019,  Fern\'{a}ndez  Sare \emph{et al.} (\cite{liure}) characterized the exponentially stable region of the thermoelastic model with Cattaneo's law and claimed that, ``{\it the change from {Fourier's law to Cattaneo's law} leads to a loss of exponential stability in most coupled systems."}  Later,  Han \emph{et al.} (\cite{han})  demonstrated that if the system maintains the same wave speed, i.e., $\sigma\tau=1$, the region of exponential stability outlined in \cite{liure} can be expanded. Furthermore, they characterized the polynomial decay rates of system \eqref{101} with two specific choices of $\alpha$. 

In this paper,  we aim to make a contribution toward the open problem mentioned above by Russell. We continue the study of Han \emph{et al.} (\cite{han}) regarding the polynomial stability of system \eqref{101} following Cattaneo's law, both with and without an inertial term, when the parameters lie outside of exponentially stable regions.
Specifically, our results reveal the influence of coupling term, thermal damping term and the inertial term (determined by $\alpha, \; \beta,\; \gamma$, respectively) on the stability and decay rate of the abstract coupled system \eqref{101}.  For the system with an inertial term, we divide the non-exponential stable region of parameters into three parts and obtain different polynomial decay rates for each part.
 On the other hand, for the system without an inertial term,  the region of parameters is two-dimensional and divided into two subregions based on spectral analysis. Similarly, we obtain polynomial decay rates for each of these subregions.  Furthermore, we prove the optimality of all obtained polynomial decay rates. Our main approach relies on the frequency domain characterization developed by Borichev and Tomilov (\cite{bobtov1}), along with interpolation inequalities and detailed spectral analysis.

The rest of the paper is organized as follows. In Section \ref{sec3}, we introduce the problem and outline the main results. The proofs of polynomial stability for the system with and without an inertial term are given in \Cref{8.8} and \Cref{8.7}, respectively. In Section \ref{8.6},   we show the optimality of the decay rates in each subregion of parameters. Finally, in Section \ref{8.5}, we apply our results to  several  partial differential equation systems and obtain optimal polynomial decay rates.

\section{Preliminaries and Main Results}\label{sec3}
\setcounter{equation}{0}
In this section, we state some preliminaries and the main results of this paper.
To describe system \eqref{101} as an abstract first-order evolution equation, we define a Hilbert space
$$\mathcal{H}:=\mathcal{D}(A^{\frac{1}{2}})\times \mathcal{D}(A^{\frac{\gamma}{2}})\times H\times H,$$
which is equipped with the inner product
$$ \langle U_1,U_2\rangle_{\mathcal{H}}=\sigma(A^{\frac{1}{2}}u_1,A^{\frac{1}{2}}u_2)+m(A^{\frac{\gamma}{2}}v_1, A^{\frac{\gamma}{2}}v_2)+(v_1,v_2)+(\theta_1, \theta_2)+\tau (q_1,q_2),$$
where
$U_{i}=(u_{i},v_{i},\theta_{i},q_{i})^\top\in\mathcal{H},\;i=1,2$. Let $t>0$, {define $v(t) := u_t(t)$,
$U(t):=(u(t), v(t), \theta(t), q(t))^\top$, and $U_0:=(u_0, u_1, \theta_0 , q_0)^\top \in \mathcal{H}$,} then system \eqref{101} can be written as an evolution equation  in $\mathcal{H}$:
\begin{equation}\label{105}
\left\{ \begin{array}{lll}
\displaystyle {\mathrm{d}\over \mathrm{d} t}  U(t)=\mathcal{A} U(t) ,\quad t>0,\\ \noalign{\medskip}  \displaystyle
U(0)=U_0, \end{array} \right.
\end{equation}
where the operator $\mathcal{A}:{\cal D}({\cal A}) \subset \mathcal{H}\to \mathcal{H}$ is defined by
$$\mathcal{A}\left[\begin{array}{c}
u\\
v\\
\theta\\
q
\end{array} \right]=\left(\begin{array}{c}
v\\
-(I+mA^{\gamma})^{-1}(\sigma A  u-A^{\alpha}\theta)\\
-A^\alpha v+A^{\frac{\beta}{2}}q\\ \displaystyle
{1\over \tau}(-q-A^{\frac{\beta}{2}}\theta)
\end{array}\right),
$$
with domain
$$\mathcal{D}(\mathcal{A})=\left\{ (u, v, \theta, q)^{\top}\in \mathcal{H}\,\left|\,
\begin{array}{l}
v\in \mathcal{D}(A^{1\over2}),\;   \sigma A  u-A^\alpha\theta \in \mathcal{D}(A^{ -\frac{\gamma}{2}}),\\
-A^\alpha v+A^{\frac{\beta}{2}}q\in H,\;
q+A^{\frac{\beta}{2}}\theta\in H
\end{array}
\right.\right\}.$$
It is clear that
\begin{align}\label{5.27}
    {\cal D}({\cal A}) \supset {\cal D}_0 \doteq {\cal D}(A^{1-{\gamma\over2}})
\times {\cal D}(A^{{1\over2} \vee \alpha}) \times {\cal D}(A^{(\alpha -{\gamma\over2})\vee {\beta\over2}}) \times {\cal D}(A^{\beta\over2}),
\end{align}
where $\alpha \vee \beta = \max\{\alpha, \; \beta\}$ for $\alpha,\; \beta \in \mathbb{R}.$

The well-posedness of the system \eqref{105} is stated as follows.
\begin{lemma}\label{t-1-1}
	Let $\mathcal{A}$ and $\mathcal{H}$ be defined as above and $(\alpha, \beta, \gamma)\in E$.
	If $\beta\geq 2\alpha-1$, then $0\in \rho(\mathcal{A})$  and $\mathcal{A}$ generates a $C_{0}$-semigroup $e^{{\cal A}t }$ of contractions on  $\mathcal{H}$.
\end{lemma}

\begin{proof}  	
	 It is easy to see that $\mathcal{A}$ is a densely defined operator {in} $\mathcal{H}$  by \eqref{5.27} and the density of $\mathcal{D}_0$ in $\mathcal{H}$. Moreover,
	$$\Re \langle\mathcal{A}U,U \rangle_{\mathcal{H}}=-\|q\|^2\leq 0,\;\forall \; U\in\mathcal{D}(\mathcal{A}),$$
	then $\mathcal{A}$ is dissipative in $\mathcal{H}$.
	
	  We claim that $\mathcal{A}$ is a bijection, then $\mathcal{A}^{-1}$ exists.
	In fact,  for any given $G:=(g_1,g_2,g_3,g_4)^{\top}\in\mathcal{H}$, let
	$U=(u,v,\theta,q)^{\top}$, where
	\begin{align*}
	v =g_1, \quad
	\theta =-A^{-\frac{\beta}{2}}(\tau g_4+A^{-\frac{\beta}{2}}(g_3+A^{\alpha}g_1)), \quad   q =A^{-\frac{\beta}{2}}(g_3+A^{\alpha}g_1),
	\end{align*}
	and $$
	u =\sigma^{-1}A^{-1}(A^{\alpha}\theta-(I+mA^{\gamma})g_2),
	$$
	
	\noindent
	then $U \in \mathcal{D}(\mathcal{A}) $ and $\mathcal{A} U=G$. This implies that  ${\cal A}$ is an injective as well.
	
	Now we show the boundedness of $\mathcal{A}^{-1}$.  Note that $\mathcal{D}(A^{\frac{1}{2}})\subset \mathcal{D}(A^{\frac{\gamma}{2}})$, we have
	\begin{equation}\label{202}
	\|v\|_{\mathcal{D}(A^{\frac{\gamma}{2}})}\leq\|g_1\|_{\mathcal{D}(A^{\frac{1}{2}})}\leq C\|G\|_{\mathcal{H}}.
	\end{equation}
	Recalling $\beta\ge2\alpha-1$, we get
	\begin{equation}\label{203}
	\|q\|\leq C(\|g_3\|+\|A^{\alpha-\frac{\beta}{2}}g_1\|)
	\leq C(\|g_3\|+\|g_1\|_{\mathcal{D}(A^{\frac{1}{2}})})\leq C\|G\|_{\mathcal{H}},
	\end{equation}
	\begin{equation}\label{204}
	\|\theta\|\leq C(\|g_4\|+\|q\|)\leq C\|G\|_{\mathcal{H}},
	\end{equation}
	and
	\begin{eqnarray}\label{205}
 \begin{array}{lll}
    	\|A^{\frac{1}{2}}u\|
	&\leq &
	C\big(\|A^{-\frac{1}{2}+\alpha-\frac{\beta}{2}}g_4\|+  \|A^{-\frac{1}{2}+\alpha-\beta}g_3\|
	+ \|A^{2\alpha-\beta-\frac{1}{2}}g_1\| +  \|A^{-\frac{1}{2}}(I+mA^\gamma)g_2\|\big)
	\cr
 \noalign{\medskip}
	&\leq & C\|G\|_{\mathcal{H}}.\quad
 \end{array}
	\end{eqnarray}
	\noindent
 Consequently, one obtains the boundedness of $\mathcal{A}^{-1}$ by (\ref{202})-(\ref{205}). By the Lumer-Phillips Theorem \cite{pazy}, $\mathcal{A}$ generates a $C_{0}$-semigroup $e^{{\cal A}t }$ of contractions on $\mathcal{H}$. The proof is completed.
\end{proof}

\begin{remark}\label{R-2-1}
	If $\beta< 2\alpha-1$, $0$ is a spectrum point of $\mathcal{A}$. This case will be considered in the future.
\end{remark}

Our interest is the stability properties of system \eqref{101}, especially the polynomial stability when {parameters $( \alpha, \;\beta,\; \gamma) $ satisfy certain conditions. Let us recall the corresponding definitions.}

\begin{definition}\label{def}
	Let $\mathcal{A}\,:\,\mathcal{D}(\mathcal{A}) \subset \mathcal{H} \to \mathcal{H}$ generate a   $C_0$-semigroup $e^{\mathcal{A}t}$  of contractions on a  Hilbert space  $\mathcal{H}$.
	\begin{enumerate}
		\item [{\rm (i)}] $e^{\mathcal{A}t}$ is said to be {\bf strongly stable} if
		$\underset{t\to\infty}{\lim}\|e^{\mathcal{A}t}x_0 \|_{\mathcal{H}}=0, \quad  \forall \;x_0\in \mathcal{H};$
		\item[{\rm (ii)}] $e^{\mathcal{A}t}$ is said to be {\bf exponentially stable} with decay rate $\omega>0$, if
		there exists a constant  $ C>0$   such that
		\begin{equation*}\label{def1}
		\|e^{\mathcal{A}t} \|_{{\cal L}(\mathcal{H})}\le C e^{-\omega t}, \quad  \forall \;t\ge 0;
		\end{equation*}
		\item[{\rm (iii)}]
		$e^{\mathcal{A}t}$ is said to be {\bf polynomially stable} of order $k>0$ if there exists a constant $C>0$   such that
		\begin{equation*}\label{def2}
		\|e^{\mathcal{A}t}\mathcal{A}^{-1}\|_{{\cal L}(\mathcal{H})} \le C t^{-k}, \quad  \forall \;t>0.
		\end{equation*}
	\end{enumerate}
\end{definition}

Based on the result in \cite{Arendt,huang}, the strong stability of the semigroup $e^{t\mathcal{A}}$  is equivalent to $i {\mathbb R}\subset \rho(\mathcal{A})$ since $0\in \rho(\mathcal{A})$ and $\mathcal{A}^{-1}$ is compact (\cite{han}).
The following lemma can be proved by the standard argument (see, e.g., the proof of { \cite[Theorem 2.1]{han}), we omit the details here.}

\begin{lemma}\label{l202}
	Assume 	that $(\alpha, \beta,\gamma)\in E$   satisfies $\beta\geq 2\alpha-1$.	
	Then  $i {\mathbb R}\subset \rho(\mathcal{A})$  for $m\ge0$, and consequently,  the semigroup $e^{t\mathcal{A}}$ is strongly stable.
\end{lemma}

We further discuss the explicit decay rates of the system with $(\alpha,\beta,\gamma)\in E$. To begin with, we divide the region of the parameter $E$ into the following subregions according to the analysis for the spectrum of ${\cal A}$ in Section \ref{8.6}.
\begin{definition}
	\label{def:E-partition}
	The parameters region  $E \!=\! \curly{\!( \alpha,\beta,\gamma)\! \mid \!\alpha\! \in [0,1],\! \beta\! \in [0,1],\! \gamma\! \in (0,1] }$ can be partitioned into the following non-overlapping subregions
	\allowdisplaybreaks
	 \begin{align*}
	&T_1 = \curly{(\alpha,\beta,\gamma) \in E \;\Big|\; \frac{1}{2}<\alpha <\frac{\beta +1}{2},\;0<\beta \leq 1,\; 0 < \gamma <2 \alpha -\beta} , \\
	&T_2 = \curly{(\alpha,\beta,\gamma) \in E \;\Big|\; 0 \leq \alpha <\frac{1}{2},\;0<\beta <1,\; 0 < \gamma <1-\beta}, \\
	&T_3 = \curly{(\alpha,\beta,\gamma) \in E \;\Big|\; 0\leq \alpha <\frac{\beta +\gamma }{2},\;1-\gamma <\beta \leq 1,\;0<\gamma \le 1}, \\
	&T_4 = \curly{(\alpha,\beta,\gamma) \in E \;\Big|\; \frac{\beta +1}{2}<\alpha \leq 1,\;0\leq \beta <1,\;0 < \gamma \leq 1}, \\
	&F_{12} = \curly{(\alpha,\beta,\gamma) \in E \;\Big|\; \alpha =\frac{1}{2},\;0<\beta <1-\gamma ,\;0 < \gamma <1}, \\
	&F_{13} = \curly{(\alpha,\beta,\gamma) \in E \;\Big|\; \alpha =\frac{\beta +\gamma }{2},\;1-\gamma <\beta \leq 1,\; 0 < \gamma <1}, \\
	&F_{14} = \curly{(\alpha,\beta,\gamma) \in E \;\Big|\; \alpha =\frac{\beta +1}{2},\;0<\beta \leq 1,\;0 < \gamma <1}, \\
	&F_{2} = \curly{(\alpha,\beta,\gamma) \in E \;\Big|\; 0\leq \alpha <\frac{1}{2},\;\beta =0,\;0 < \gamma <1}, \\
	&F_{23} = \curly{(\alpha,\beta,\gamma) \in E \;\Big|\; 0\leq \alpha <\frac{1}{2},\;\beta =1-\gamma ,\;0<\gamma <1}, \\
	& L_{123} =\curly{(\alpha,\beta,\gamma) \in E \;\Big|\;
		\alpha =\frac{1}{2},\;\
		\beta =1-\gamma,\;\
		0<\gamma <1}, \\
	&L_{124} = \curly{(\alpha,\beta,\gamma) \in E \;\Big|\;
		\alpha =\frac{1}{2},\
		\beta =0,\
		0 < \gamma <1},\; \\
	& L_2 =\curly{(\alpha,\beta,\gamma) \in E \;\Big|\; 0\leq \alpha <\frac{1}{2},\;\beta =0,\;\gamma =1},\; \\
	&L_{34} = \curly{(\alpha,\beta,\gamma) \in E \;\Big|\; \alpha =\frac{\beta +1}{2},0<\beta \leq 1,\;\gamma =1}, \\
	&P_{234} = \curly{(\alpha,\beta,\gamma) \in E \;\Big|\; \alpha =\frac{1}{2},\;\beta =0,\;\gamma =1}.
	\end{align*}
\end{definition}
\begin{remark}
	{From the definition, it is clear that planes $F_{ij} $ are (parts) of the intersections of $T_i$ and $T_j$, where $i,j\in\{1,2,3,4\},\; i\neq j$,  lines $L_{ijk}$  are intersections of $T_i,\;T_j,\; T_k,$ where $ i,j,k \in\{1,2,3,4\},\; ,i\neq j\neq k$.
	Furthermore, the plane $F_2$ is a part of the boundary of $T_2$,
	the line $L_2$ is a special part of the boundary of $F_2$,
	the line $L_{34}$ is a  part of the boundary of $F_{14}$, and $P_{234}$ is a point.}
	The regions in {Definition \ref{def:E-partition} } are visualized in Figure \ref{fig:E}.
\end{remark}
	
\begin{remark}\label{rem23}	According to \cite[Theorem 2.1]{han}, it was proved that the semigroup is exponentially stable on plane  $F_{13}$, along the lines $L_{123}$, $L_{34}$ and at the point $P_{234}$.
	In the case of region $T_{4},$ where $0$ is a spectrum, we will analyze it in the future work. Moreover, the asymptotic behavior on $L_2$  remains unclear at present (as indicated by the analysis in  Section \ref{8.6}).
\end{remark}

\begin{figure}
	\centering
	\includegraphics[width=6cm]{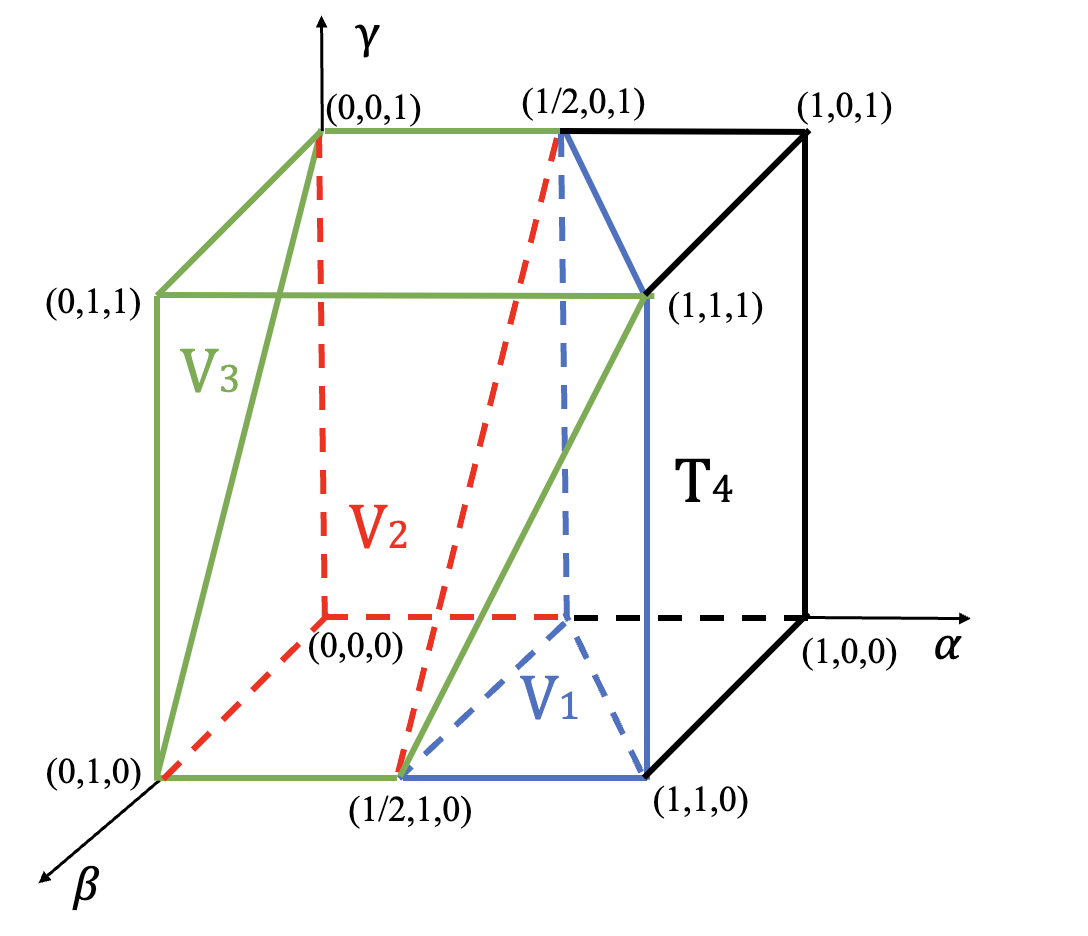}
	\caption{\small   Regions of polynomial stability when $m>0$.}
	\label{m>0}
\end{figure}

\begin{figure}[t]
	\centering
	\begin{subfigure}{0.4\textwidth}
		\centering
		\includegraphics[scale=0.6]{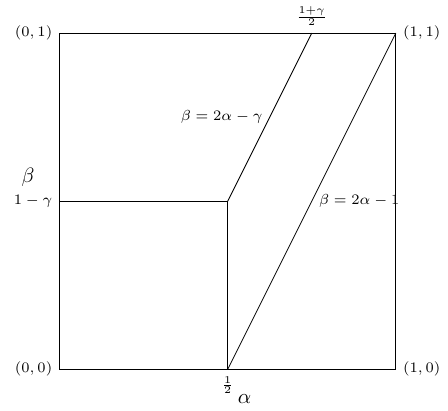}
		\caption{$\gamma=\frac{1}{2}$}
	\end{subfigure}~
	\begin{subfigure}{0.4\textwidth}
		\centering
		\includegraphics[scale=0.6]{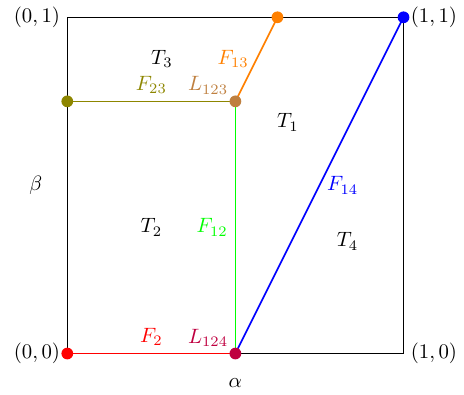}
		\caption{$\gamma=\frac{1}{4}$}
	\end{subfigure}\\
	\begin{subfigure}{0.4\textwidth}
		\centering
		\includegraphics[scale=0.6]{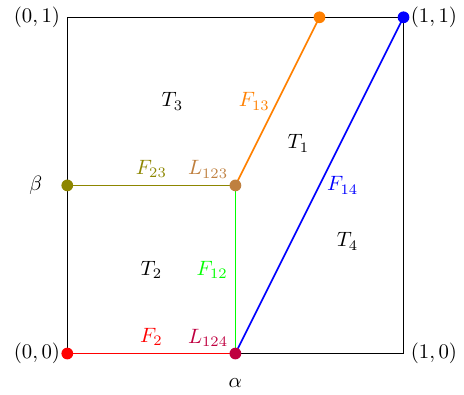}
		\caption{$\gamma=\frac{1}{2}$}
	\end{subfigure}~
	\begin{subfigure}{0.4\textwidth}
		\centering
		\includegraphics[scale=0.6]{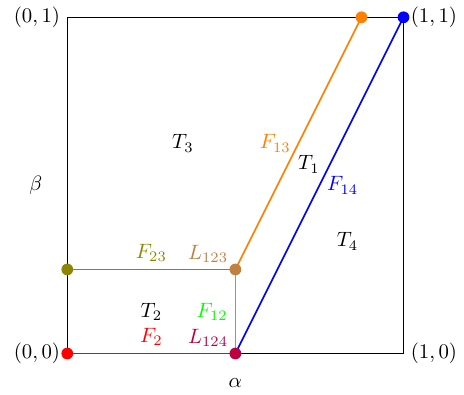}
		\caption{$\gamma=\frac{3}{4}$}
	\end{subfigure}\\
	\caption{Partition of $E$ according to \Cref{def:E-partition} when $\gamma$ is  $\frac{1}{4}$, $\frac{1}{2}$, $\frac{3}{4}$, respectively.}
	\label{fig:E}
\end{figure}

More precisely, due to Remark \ref{rem23} and spectral analysis in Section \ref{8.6}, we shall analyze the  stability and decay rate of system \eqref{101} with  inertial term ($m\not=0$) when parameters $(\alpha,\,\beta,\,\gamma)$ belong to the following three regions:
\begin{equation}
\label{par1}
\begin{array}{l}
\displaystyle
V_1:=T_1\cup F_{12}\cup F_{14}\cup L_{124}, \\ \noalign{\medskip}  \displaystyle
V_2:=T_2\cup F_{2},
\\ \noalign{\medskip}  \displaystyle
V_3:=T_{3}\cup F_{23}.
\end{array}\end{equation}

\noindent
The regions $V_i, \; i=1,2,3$ are  illustrated  in Figure \ref{m>0}. The first main result of this paper reads as follows.

\begin{theorem}\label{t-3-1}
	Let $m>0$ and $V_i,\;i=1,2,3$ be defined as in \eqref{par1}.
	Then semigroup $e^{\mathcal{A}t}$ associated with (\ref{105}) has the following stability
	properties: 
	
	\begin{enumerate}
		\item[{\rm (i)}] It is polynomially stable of order $k_1=\displaystyle \frac{2\alpha-\gamma}{2(2\alpha-\beta-\gamma)}$ in $V_1$;
		
		\item[{\rm (ii)}] It is polynomially stable of order $k_2=\displaystyle \frac{1-\gamma}{2(-2\alpha-\beta-\gamma+2)}$ in $V_2$;
		
		\item[{\rm (iii)}] It is polynomially stable of order $k_3=\displaystyle
		\frac{1-\gamma}{2(-2\alpha+\beta+\gamma)}$ in $V_3$.
	\end{enumerate}
\end{theorem}

We also consider the case in which there is no inertial term in system \eqref{101}. Similarly, we partition the region of the parameters $E^*:=[0,1]^2$ as follows.
\begin{definition}
	\label{def:E-tilde-partition}
	The parameters region  $E^* = \curly{( \alpha,\beta) \mid \alpha,\;\beta \in [0,1] }$ can be partitioned into the following non-overlapping subregions
	\allowdisplaybreaks
	\begin{align*}
	&T_1^* = \curly{(\alpha,\beta) \in E \Big| \frac{1}{2}<\alpha <\frac{\beta +1}{2},0<\beta \leq 1} ,\\
	&T_2^* = \curly{(\alpha,\beta) \in E \Big| 0\leq \alpha <\frac{1}{2},0<\beta <1},\\
	&T_4^* = \curly{(\alpha,\beta) \in E \Big| \frac{\beta +1}{2}<\alpha \leq 1,0\leq \beta <1},\\
	&F^*_{12} = \curly{(\alpha,\beta) \in E \Big| \alpha =\frac{1}{2},0<\beta <1},\\
	&F^*_{14} = \curly{(\alpha,\beta) \in E \Big| \alpha =\frac{\beta +1}{2},0<\beta \leq 1},\\
	&F^*_{2} = \curly{(\alpha,\beta) \in E \Big| 0\leq \alpha <\frac{1}{2},\beta =0},\\
	&L^*_{124} = \curly{(\alpha,\beta) \in E \Big| \alpha =\frac{1}{2},\beta =0},\\
	&L^*_{23} = \curly{(\alpha,\beta) \in E \Big|
		0\leq \alpha <\frac{1}{2},\beta =1},\\
	&P^*_{123} = \curly{(\alpha,\beta) \in E \Big| \alpha =\frac{1}{2},\beta =1}.
	\end{align*}
\end{definition}

\begin{remark}
In Figure \ref{8131}, we illustrate the regions given in  Definition \ref{def:E-tilde-partition}.
The lines $F^*_{12}, F^*_{14}$ denote the boundaries of $T^*_1$, $T^*_2$ and $T^*_1$, $T^*_4$, respectively.  The lines  $F^*_2$ and $L^*_{23}$ are  boundaries of $T^*_2$.  The point  $L^*_{124}$ is the intersection of $T_i^*, i=1,2,4$. The region $T^*_4$ is the area where $0$ is a spectrum of $\mathcal{A}$, which will be discussed in the future work.
 \end{remark}

 According to \cite[Theorem 2.2]{han}, the system is exponentially stable on the point $P_{123}^*$.
We consider the decay rate of the solution when parameters belonging to the remaining area $E^*\backslash (T_4^*\cup P^*_{123})$.
Based on the analysis presented in Section 5,
we partition this region into the
following two parts.
\begin{equation}
\label{par2}
\begin{array}{l}\displaystyle
V_1^*:=T_1^*\cup F^*_{12}\cup F^*_{14}\cup L^*_{124},
\\ \noalign{\medskip}  \displaystyle
V_2^*:=T_2^*\cup F^*_2\cup L^*_{23}.
\end{array}
\end{equation}

\begin{figure}[htbp]
	\centering
	\begin{minipage}[t]{0.4\textwidth}
		\centering
		\includegraphics[scale=0.6]{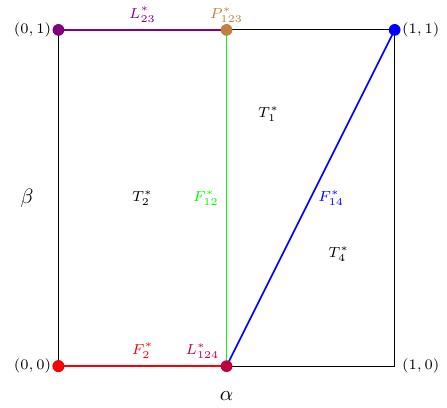}
		\caption{Partition of $E^*$ according to \Cref{def:E-tilde-partition}}
		\label{8131}
	\end{minipage}
	\begin{minipage}[t]{0.4\textwidth}
		\centering
		\includegraphics[scale=0.25]{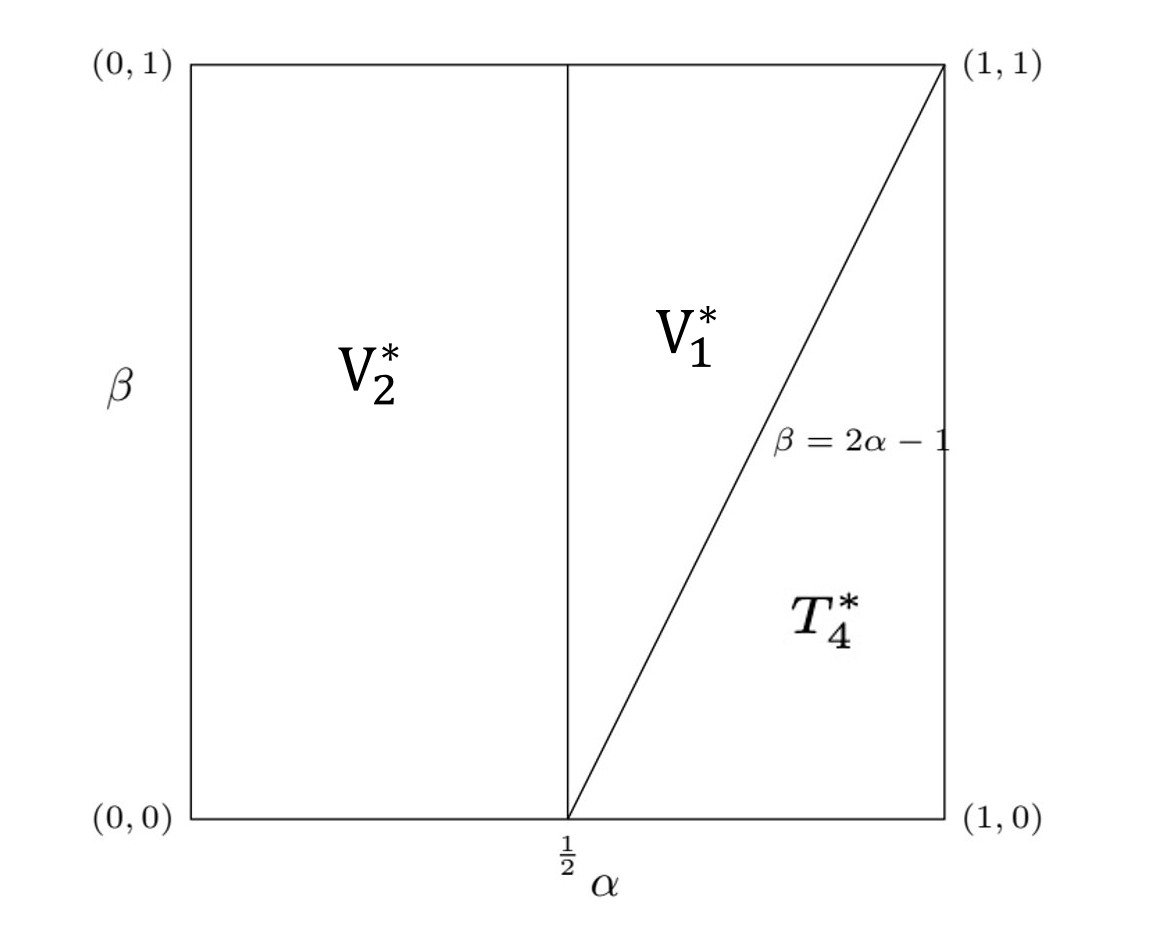}
		\caption{Partition of $V_i^*,i=1,2$}
		\label{8133}
	\end{minipage}
\end{figure}
The regions of {$V_i^*,i=1,2$}  are shown in Figure \ref{8133}.
The second main result is as follows.
\begin{theorem}\label{m00}
	Let $m=0$ and $V_i^*,i=1,2$ be defined as in \eqref{par2}.
	Then semigroup $e^{\mathcal{A}t}$ associated with (\ref{105}) satisfies
	\begin{enumerate}
		\item[{\rm (i)}]  It is polynomially stable of order $\displaystyle
		k_1^*=\frac{\alpha}{2\alpha-\beta} $ in $V_1^*$;
		
		\item[{\rm (ii)}]  It is polynomially stable of order $\displaystyle
		k_2^*=\frac{1}{2(2-2\alpha-\beta)} $ in $V_2^*$.
		
	\end{enumerate}
\end{theorem}

Furthermore, through a detailed spectral analysis, we can show the sharpness of the decay rates obtained in the above two theorems.

\begin{theorem}\label{th-o}
	The orders of polynomial decay in \Cref{t-3-1} and \Cref{m00} are all optimal.
\end{theorem}

\begin{remark}
In  Theorems \ref{t-3-1} and \ref{m00}, we provide a comprehensive analysis of polynomial stability for system \eqref{101} under varying coupling effects, thermal damping and inertial characteristics.
When $m>0$, signifying the presence of an inertial term in system \eqref{101},   we prove that the corresponding semigroup decays polynomially at different rates when parameters $ (\alpha,\; \beta,\; \gamma) $ belong to different subregions.   In the case where   there is no inertial term (i.e., $m=0$), we  establish
a relationship between the parameters $\alpha, \; \beta $ and the order of polynomial stability.
Our results complement the stability analysis of the system previously discussed in \cite{han}.
On the other hand,
it is worth noting that the decay rates obtained in  Theorem \ref{m00},
are slower when compared to the findings presented in \cite{haoliu2}. The work in \cite{haoliu2} primarily focuses on a thermoelastic system with Fourier's law while neglecting the inertial term.  This observation reveals that the thermal damping of Cattaneo's type is weaker than the Fourier's type.
\end{remark}

The proofs for  polynomial stability rely on the following lemma,
which gives necessary and sufficient conditions for the polynomial stability of $C_0$-semigroup (see \cite{bobtov1,liurao} or \cite{batty, Rozendaal} for more general case).

\begin{lemma}\label{l-2-1}
	Assume that $\mathcal{A}$ generates a $C_0$-semigroup $e^{\mathcal{A}t}$ of contraction  on a Hilbert space $\mathcal{H}$ and satisfies   $i\mathbb{R}\subset \rho(\mathcal{A})$.
	Then, $e^{\mathcal{A}t}$ is polynomially stable  with order $ k $, if and only if
	\begin{equation}\label{e2} \limsup_{|s|\to\infty}\limits s^{-{1\over k}}\|(is-\mathcal{A})^{-1}\|_{\mathcal{L}(\mathcal{H})}<\infty.
	\end{equation}
\end{lemma}

The following interpolation theorem will play a crucial role in our proof.
\begin{lemma}\label{lemma-inter}
	Let $A \,:\, {\cal D}(A) \subset H$ be self-adjoint and positive definite, $r,\;p\;,q\in \mathbb{R}$. Then
	\begin{align}\label{4.41}
	\|A^p x\| \le \|A^q x\|^{p-r\over q-r} \|A^r x\|^{q-p\over q-r},  \quad
	\forall \;  r\le p \le q, \; x\in {\cal D}(A^q).
	\end{align}
\end{lemma}
\begin{proof}
	The case $ 0\leq r\le p \le q$ can be found in  \cite{liuzheng}.
	If $  r\leq 0  $ and $0\le p \le q$,  we let   $ y\in   \mathcal{D}(A^{q-r}) ,\; x = A^{-r}y.$  It follows from  (\ref{4.41}) that
	$$
	\|A^p x\| = \|A^{p-r}y\| \le \|A^{q-r}y\|^{{p-r\over q-r}}\|  y\|^{q-p\over q-r}=\|A^q x\|^{p-r\over q-r} \|A^r x\|^{q-p\over q-r}.
	$$
	The other cases can be proved by the same argument.
\end{proof}

\section{Polynomial Stability of the System  with an Inertial Term (Proof of  \Cref{t-3-1})}\label{8.8}
\setcounter{equation}{0}
This section is devoted to analyzing the polynomial stability of system \eqref{101} with $m>0$.
Due to   {Lemmas}  \ref{l202} and \ref{l-2-1}, it is sufficient to show that there exists a constant $r > 0$
such that
\begin{equation}\label{huang2}
\inf \big\{
\lambda^{k}\|i\,\lambda U- {\cal A} U\|_{\cal H} \,\big|\quad
\|U\|_{\cal H}=1,\: \lambda\in{\mathbb R} \big\} \geqslant  r
\end{equation}
holds for some $k>0$.

 By contradiction, we suppose that \eqref{huang2} fails.     Then,
there at least  {exists}  one sequence $\{ \lambda_{n},  U_{n}  \}_{n=1}^\infty$, where $\lambda_{n}\in {\mathbb R} ,\;  U_n:=(u_{n},\, {v}_{n},\, \theta_{n},\, q_{n})^\top\in \mathcal{D}({\cal A})$ such that
\begin{equation}\label{unitnorm}
\|U_{n}\|_{\cal H}  =\big\| (u_{n},\, {v}_{n},\, \theta_{n},\, q_{n})^\top\big\|_{ \cal H }  =1,
\end{equation}
and as $n\to\infty$,
\begin{align}&
\lambda_{n}\to\infty,\nonumber
\\ \noalign{\medskip} & \displaystyle
\label{331}
\lambda_{n}^{k}\big( i\, \lambda_{n} U_{n}- {\cal A}   U_{n} \big) =o(1)
\quad  \;  \mbox{ in  }  \; {\cal H} ,
\end{align}
i.e.,
\begin{eqnarray}
&&\lambda_n^k(i\lambda_nA^{\frac{1}{2}}u_n- A^{\frac{1}{2}}v_n)=o(1)\quad  \;  \mbox{ in  }  \;  H,\label{s1+}\\   \noalign{\medskip}  \displaystyle
&&\lambda_n^kA^{-\frac{\gamma}{2}}\left(i\lambda_nv_n+i\lambda_n m A^{\gamma}v_n+\sigma A u_n-A^{\alpha}\theta_n\right)=o(1)\quad  \;  \mbox{ in  }  \;  H,\label{s2+}\\\noalign{\medskip}  \displaystyle
&&\lambda_n^k(i\lambda_n\theta_n-A^{\frac{\beta}{2}}q_n+A^\alpha v_n)=o(1)\quad  \;  \mbox{ in  }  \;  H,\label{s3+}\\\noalign{\medskip}  \displaystyle
&&\lambda_n^k(i\lambda_n\tau  q_n+ q_n+A^{\frac{\beta}{2}}\theta_n)=o(1)\quad  \;  \mbox{ in  }  \;  H.\label{s4+}
\end{eqnarray}

\begin{lemma}\label{lemma-1}
	For $m\ge 0$, one has
	\begin{eqnarray}\label{v1_0}
	&&	\|q_n\|= \lambda_n^{-\frac{k}{2}} o(1),
	\\ \noalign{\medskip}  \displaystyle
	\label{v1_8}
	&& \| A^{\frac{\beta}{2}}\theta_n\|=\lambda_n^{1-\frac{k}{2}}o(1),
	\\ \noalign{\medskip}  \displaystyle
	\label{+211}
	&&	\|A^{\frac{1}{2}}u_n\|,\; \|v_n\|,\;   \|\theta_n\|,\; \|\lambda_n^{-1}A^{\frac{1}{2}}v_n\|
	=O(1).
	\end{eqnarray}
	Furthermore, when $m>0,$ it holds
	\begin{equation}
	\label{+211-1}
	\|A^{\gamma \over2}v_n\|,\;\;\|\lambda_n A^{\gamma \over2}u_n\|
	=O(1) .
	\end{equation}
\end{lemma}
\begin{proof}
	We can obtain \eqref{v1_0} {by \eqref{331} and the fact that $\mathcal{A}$ is  dissipative.} Then, combining
	\eqref{s4+} and \eqref{v1_0} yields \eqref{v1_8}. \eqref{+211} and \eqref{+211-1} are clear from \eqref{unitnorm} and \eqref{s1+}.
\end{proof}

{Now, we shall show (i)--(iii) in \Cref{t-3-1}, respectively.}

\textbf{ (i) } Assume $(\alpha, \, \beta,\, \gamma) \in V_1$ and let $  k = \frac{2( 2\alpha-\beta-\gamma)}{2\alpha-\gamma}$.

From  \eqref{s3+}  and the assumption    $ - \alpha+ \frac{\gamma}{2}<-{ \beta\over 2}\leq 0$,  we have   that
$$ i\lambda_nA^{-\alpha+\frac{\gamma}{2}}\theta_n-A^{-\alpha+\frac{\gamma}{2}+\frac{\beta}{2}}q_n+A^{\frac{\gamma}{2}} v_n =o(1).
$$
Combining these with {\eqref{v1_0} and \eqref{+211-1} } yields
\begin{equation}\label{v1_1}
\|A^{-\alpha+\frac{\gamma}{2}}\theta_n\|=\lambda_n^{-1}O(1).
\end{equation}
By using Lemma \ref{lemma-inter}, \eqref{v1_8}, \eqref{v1_1} and the {choice} of $k$ given above, one has
\begin{equation}\label{v1_2}
\left\| \theta_n\right\| \leq \|A^{\frac{\beta}{2}} \theta_n\|^{\frac{2\alpha-\gamma}{2\alpha+\beta-\gamma}} \| A^{-\alpha+\frac{\gamma}{2}} \theta_n \| ^{\frac{\beta}{2\alpha+\beta-\gamma}}=o(1).
\end{equation}

\noindent Note that $A^{-\alpha+\frac{\gamma}{2}}\theta_{n}$ and $A^{-\alpha}(I+mA^{\gamma})v_n$ are bounded due to
$-\alpha+\frac{\gamma}{2}<0$ and \eqref{+211-1}. Then, {taking the inner product of \eqref{s2+} with $\lambda_n^{-k}A^{-\alpha+\frac{\gamma}{2}}\theta_{n}$, \eqref{s3+} with $\lambda_n^{-k}A^{-\alpha}(I+mA^{\gamma})v_n$, respectively,} we have
\begin{equation*}\label{v1_3}
(\theta_{n},i\lambda_nA^{-\alpha}(I+{m}A^{\gamma})v_n)+\sigma(A^{1-\alpha}\theta_n,u_n)-(\theta_{n},\theta_{n})=o(1),
\end{equation*}
\begin{equation*}\label{v1_4}
(i\lambda_n\theta_n, A^{-\alpha}(I+mA^{\gamma})v_n)-(A^{\frac{\beta}{2}-\alpha}(I+{m}A^{\gamma})q_n, v_n)+((I+mA^{\gamma})v_n, v_n)=o(1).
\end{equation*}
Adding the above two equations implies
\begin{equation}\label{v1_5}
\sigma(A^{1-\alpha}\theta_n,\, u_n)-(A^{\frac{\beta}{2}-\alpha}(I+mA^{\gamma})q_n,\,  v_n) -\|\theta_{n}\|^2+\|v_n\|^2+m\|A^{\frac{\gamma}{2}}v_n\|^2=o(1).
\end{equation}
Combining {$ \alpha>\max\{{1\over2}, {\beta+\gamma \over 2}\}$} and  \eqref{v1_0}, \eqref{+211}, \eqref{+211-1}, \eqref{v1_2}, one has
\begin{equation}
\label{vq}
(A^{{1\over2}-\alpha}\theta_n,\, A^{1\over2} u_n),\;\;
(A^{\frac{\beta+\gamma}{2}-\alpha}q_n,\, A^{-{\gamma\over 2}}(I+mA^{\gamma}) v_n) = o(1).
\end{equation}
Substituting \eqref{vq} into \eqref{v1_5}, {along with (\ref{v1_2})}, yields
\begin{equation}\label{v1_10}
\|v_n\|,~~\|A^{\frac{\gamma}{2}}v_n\| =o(1).
\end{equation}

\noindent Moreover, taking  the inner product of \eqref{s2+}  with ${\lambda_n^{-k}}A^{\frac{\gamma}{2}}u_n$ on $H$, we get
\begin{equation}\label{v1_6}
(i\lambda_nv_n, u_n) +(i\lambda_n{m}A^{\gamma}v_n, u_n) +\sigma (A u_n, u_n) -(A^{\alpha}\theta_n, u_n)\to 0,
\end{equation}
which together with \eqref{s1+} implies
\begin{equation}\label{316-}
-\|v_n\|^2-m\|A^{\frac{\gamma}{2}}v_n\|^2 +\sigma\|A^{\frac{1}{2}} u_n\|^2 -(A^{\alpha}\theta_n, u_n)
= o(1).
\end{equation}
{Thanks to} $\alpha -{\beta\over2} \le {1\over2} $, $\eqref{s1+}$ and \eqref{v1_8}, {we get}
\begin{eqnarray}\label{m0191}
(A^{\alpha}\theta_n,\, u_n)
&=& (A^{\frac{\beta}{2}}\theta_n,\, (i\lambda_n)^{-1}A^{\alpha-\frac{\beta}{2}} v_n + \lambda_n^{-1-k}o(1))\cr
&=& (A^{\frac{\beta}{2}}\theta_n,\, (i\lambda_n)^{-1}A^{\alpha-\frac{\beta}{2}} v_n) +{\lambda_n^{-\frac{3k}{2}}}o(1).
\end{eqnarray}
By $\eqref{s3+}$, {we know}
\begin{equation}\label{m019}
A^{\alpha-\frac{\beta}{2}} v_n= -i\lambda_nA^{-\frac{\beta}{2}}\theta_n+ q_n+ \lambda_n^{ -k}o(1).
\end{equation}
Thus, by substituting \eqref{m019} into \eqref{m0191} and using \eqref{v1_0}, \eqref{v1_8}, \eqref{v1_2}, we conclude that
\begin{equation}\label{m0192}
(A^{\alpha}\theta_n,\, u_n)=o(1).
\end{equation}
Combining \eqref{v1_10}, \eqref{316-} and \eqref{m0192} yields
\begin{equation}\label{m018}
\|A^{\frac{1}{2}}u_n\|={o(1)}.
\end{equation}
In summary, by \eqref{v1_0}, \eqref{v1_2}, \eqref{v1_10} and \eqref{m018}, we obtain that $\|U_n\|_{\mathcal{H}}=o(1)$, which contradicts the assumption \eqref{unitnorm}. The proof of   (i) is completed.


\vskip 4mm

\textbf{(ii)} Assume $(\alpha,\; \beta,\; \gamma)\in V_2$, and let $k= \frac{2(-2\alpha-\beta-\gamma+2)}{1-\gamma}$.

We first show that for  $\alpha <{1\over2}$,
\begin{equation}
\label{vv3}
\|A^{-\frac{1}{2}}(I+mA^{\gamma})v_n\| =\lambda_n^{-1}O(1).
\end{equation}
In fact, it follows from \eqref{s2+} that
\begin{equation}
\label{vv2}
\|\lambda_nA^{-\frac{1}{2}}(I+mA^{\gamma})v_n\| \le \sigma \|A^{1\over2} u_n\| +\|A^{-{1\over2}+\alpha}\theta_n\| + \lambda_n^{-k}o(1).
\end{equation}
Thus, \eqref{vv3} follows from \eqref{+211} and \eqref{vv2}.

Taking the inner product of \eqref{s3+} with $\lambda_n^{-1-k}\theta_n$ on $H$ yields
\begin{equation}\label{v2_0}
i\|\theta_n\|^2  -\lambda_n^{-1}(A^{\frac{\beta}{2}}q_n,\theta_{n})+\lambda_n^{-1}(A^{\alpha}v_n, \theta_{n})=\lambda_n^{-1-k}o(1).
\end{equation}
From  \eqref{v1_0} and \eqref{v1_8}, it is easy to see that
\begin{equation}
\label{qth}
| \lambda_n^{-1}(A^{\frac{\beta}{2}}q_n, \theta_n)| \leq\|q_n\|\|\lambda_n^{-1}A^{\frac{\beta}{2}}\theta_n\|=\lambda_n^{-k}o(1).
\end{equation}

\noindent
To deal with the  third term of \eqref{v2_0}, we use the Cauchy-Schwarz inequality and  \eqref{v1_8}  to get
\begin{equation}
\label{a}
|\lambda_n^{-1} (A^{\alpha}v_n,\theta_{n})| \leq \|A^{\alpha-\frac{\beta}{2}}v_n\|\|\lambda_n^{-1}A^{\frac{\beta}{2}}\theta_{n}\|
={\lambda_n}^{-\frac{k}{2}}\|A^{\alpha-\frac{\beta}{2}}v_n\|o(1) .
\end{equation}

\noindent
Furthermore, by Lemma \ref{lemma-inter},  \eqref{+211}, \eqref{+211-1},  \eqref{vv3} {and $\beta+\gamma<1$}, we deduce that
\begin{equation}
\label{vv}
\|A^{\alpha-\frac{\beta}{2}}v_n\| \leq
\left\{
\begin{array}{ll}
O(1) & \;  \mbox{when}  \; \; \displaystyle   0\leq \alpha\leq\frac{ \gamma}{2},
\\ \noalign{\medskip}  \displaystyle
{\|A^{\gamma-\frac{1}{2}}v_n\|^{1- \frac{2\alpha-\beta-2\gamma+1}{2(1-\gamma)}}\|A^{\frac{1}{2}}v_n\|^{\frac{2\alpha-\beta-2\gamma+1}{2(1-\gamma)}} =\lambda_n^{\frac{2\alpha-\beta-\gamma}{1-\gamma}}O(1)} & \;  \mbox{when}  \; \; \displaystyle    \frac{ \gamma}{2}<\alpha<{1\over2}.
\end{array}
\right.
\end{equation}

\noindent
Substituting \eqref{qth}-\eqref{vv} into \eqref{v2_0} and noticing the choice of $k$, we  have
\begin{equation}
\label{th1}
\|\theta_n\| = \left\{
\begin{array}{ll}
\lambda_n^{-{k\over4}}o(1) & \;  \mbox{when}  \; \; \displaystyle  0\leq \alpha\leq\frac{ \gamma}{2},
\\ \noalign{\medskip}  \displaystyle
\lambda_n^{\frac{2\alpha-1}{1-\gamma}}o(1) & \;  \mbox{when}  \; \; \displaystyle
\frac{ \gamma}{2}<\alpha<\frac{1}{2}.
\end{array}
\right.
\end{equation}

\noindent
Consequently, one has
\begin{equation}
\label{th2}
\|\theta_n\| = o(1).
\end{equation}
Now we are going to prove that
\begin{equation}
\label{vv1}
\|v_n\|, \;\; \|A^{\frac{\gamma}{2}}v_n\| =o(1).
\end{equation}

\noindent
Based on the estimation \eqref{th1}, we shall obtain \eqref{vv1} by considering the following two cases, respectively (Case ii-A and B).

{\bf Case ii-A.}  Assume $ \alpha\leq\frac{\gamma}{2}$.

Note that $\|A^{\frac{\beta-1}{2}}(I+ mA^{\gamma})v_n\| \le C\|A^{\frac{\beta+\gamma-1}{2}} A^{\gamma\over2}v_n\| =O(1) $ due to $\beta+\gamma<1$ and \eqref{+211-1}.
Then, we can take the inner product of \eqref{s3+} with ${\lambda_n^{-k}}A^{\frac{\beta-1}{2}}(I+mA^{\gamma})v_n$ to get
\begin{equation}\label{v2_3}
\begin{array}{l}
-(A^{\frac{\beta}{2}}\theta_n,i\lambda_nA^{-\frac{1}{2}}(I+mA^{\gamma})v_n) -( q_n,A^{\frac{2\beta-1}{2}}(I+mA^{\gamma})v_n)
\\ \noalign{\medskip}  \displaystyle +\|A^{\frac{2\alpha+\beta-1}{4}} v_n\|^2+m\|A^{\frac{2\alpha+2\gamma+\beta-1}{4}} v_n\|^2=\lambda_n^{-k}o(1).
\end{array}
\end{equation}

\noindent
By \eqref{+211} {and $\beta+\gamma<1$}, we have $$\|A^{\frac{2\beta-1}{2}}(I+mA^{\gamma})v_n\| \le C \|A^{ \beta+ \gamma -1} A^{{1\over2}}v_n\| =\lambda_n O(1).$$
{Then,} substituting {the above} and \eqref{v1_0}, \eqref{v1_8}, \eqref{vv3} into \eqref{v2_3}, we have
\begin{equation}\label{v2_4}
\|A^{\frac{2\alpha+\beta-1}{4}} v_n\|,~~\|A^{\frac{2\alpha+2\gamma+\beta-1}{4}} v_n\|=\lambda_n^{\frac{1}{2}-\frac{k}{4}}o(1)=\lambda_n^{\frac{2\alpha+\beta-1}{2(1-\gamma)}}o(1).
\end{equation}

\noindent  Thanks to $ \alpha\leq \frac{\gamma}{2}$, we easily obtain  $\frac{2\alpha+2\gamma+\beta-1}{4}<\frac{\gamma}{2}<\frac{1}{2}$, and thus, by interpolation, one gets
\begin{equation}\label{v2_5}
\|A^{\frac{\gamma}{2}}v_n\|\leq\|A^{\frac{2\alpha+2\gamma+\beta-1}{4}} v_n\|^{1-\frac{1-\beta-2\alpha}{3-\beta-2\alpha-2\gamma}}\|A^{\frac{1}{2}}v_n\|^{\frac{1-\beta-2\alpha}{3-\beta-2\alpha-2\gamma}}.
\end{equation}
Therefore, we obtain \eqref{vv1} by combining \eqref{+211}, \eqref{v2_4} and \eqref{v2_5}.

\noindent

{\bf Case ii-B.}  Assume $\frac{\gamma}{2}<\alpha<{1\over2}.$

Using  Lemma \ref{lemma-inter}, \eqref{+211-1} and  \eqref{vv3} yields
\begin{eqnarray}\label{v2_8} \|A^{-\alpha}(I+mA^{\gamma})v_n\|&\leq&
\|A^{-\frac{1}{2}}(I+mA^{\gamma})v_n\|^{\frac{2\alpha-\gamma}{1-\gamma}}
\|A^{-\frac{\gamma}{2}}(I+mA^{\gamma})v_n\|^{1-\frac{2\alpha-\gamma}{1-\gamma}}\crr
&=&\lambda_n^{\frac{-2\alpha+\gamma}{1-\gamma}}O(1).
\end{eqnarray}
Taking  the inner product of \eqref{s3+} with $  \lambda_n^{-k} A^{-\alpha}(I+mA^{\gamma})v_n$, we get
\begin{equation}\label{v2_7}
(i\lambda_n\theta_n, A^{-\alpha}(I+mA^{\gamma})v_n)-(q_n, A^{\frac{\beta}{2}-\alpha}(I+mA^{\gamma})v_n)+((I+mA^{\gamma})v_n, v_n)= o(1).
\end{equation}

\noindent
For the first term of \eqref{v2_7}, one can deduce from \eqref{th1} and \eqref{v2_8} that
\begin{equation}
\label{th3}(i\lambda_n\theta_n, A^{-\alpha}(I+mA^{\gamma})v_n)=o(1).
\end{equation}
\noindent
To estimate  the second term of \eqref{v2_7}, we notice that if $\frac{\gamma}{2}<\alpha\leq\frac{\beta+\gamma}{2}$,
\begin{equation}
\label{vv4}
\|A^{-\alpha+\frac{\beta}{2}+\gamma}v_n\|\leq
\|A^{\frac{\gamma}{2}}v_n\|^{1-\frac{-2\alpha+\beta+\gamma}{1-\gamma}}
\|A^{\frac{1}{2}}v_n\|^{\frac{-2\alpha+\beta+\gamma}{1-\gamma}}
=\lambda_n^{\frac{-2\alpha+\beta+\gamma}{1-\gamma}}O(1),
\end{equation}
where we have used {$\beta+\gamma<1$, Lemmas \ref{lemma-inter} and \ref{lemma-1}}. Therefore, from \eqref{v1_0}, \eqref{+211} and \eqref{vv4}, we can deduce that
\begin{equation}
\label{vv5}
|(q_n, A^{\frac{\beta}{2}-\alpha}(I+mA^{\gamma})v_n)|
\le \left\{
\begin{array}{ll}
C\|q_n\|\|A^{-\alpha+\frac{\beta}{2}+\gamma} v_n\|=\lambda_n^{2(\beta+\gamma-1)\over 1-\gamma}o(1),
& \displaystyle
{\gamma\over 2}<\alpha\leq\frac{\beta+\gamma}{2},
\\ \noalign{\medskip}  \displaystyle
\|A^{-\alpha+\frac{\beta+\gamma}{2}}q_n\|\|A^{-\frac{ \gamma}{2}}(I+mA^{ \gamma})v_n\|=\lambda_n^{-{k\over 2}}o(1),
& \displaystyle
\frac{\beta+\gamma}{2} < \alpha <{1\over2}.
\end{array}\right.
\end{equation}

\noindent
Therefore,  we get  \eqref{vv1} by substituting \eqref{th3} and \eqref{vv5} into \eqref{v2_7}.

Finally, similar to (i), we can show that the fourth term in \eqref{316-} tends to {0}.  Indeed,  we have $|(A^{\alpha} \theta_n,  u_n)|
\le \|A^{\alpha-\frac{1}{2}}\theta_n\| \|A^{\frac{1}{2}}u_n\|=o(1)$ thanks to \eqref{+211}, \eqref{th2} and $\alpha<\frac{1}{2}$.
Thus, by (\ref{316-}), along with (\ref{vv1}), we get $\|A^{\frac{1}{2}}u_n\| =o(1)$. Therefore, combining this with \eqref{th2}, \eqref{vv1}, one can arrive at the contradiction  $\|U_n\|_{\mathcal{H}}=o(1)$.
\vskip 4mm

{\bf{(iii)}} Assume $(\alpha,\beta,\gamma)\in V_3$, and let $k=\displaystyle
\frac{2(-2\alpha+\beta+\gamma)}{1-\gamma}$.

Taking the inner product of \eqref{s3+} with $\lambda_n^{-k}\theta_n$ yields
\begin{equation}\label{v3_0}
i\|\theta_n\|^2 -\lambda_n^{-1}(A^{\frac{\beta}{2}}q_n,\theta_{n})+\lambda_n^{-1}(A^{\alpha}v_n, \theta_{n})=\lambda_n^{-1-k}o(1).
\end{equation}

\noindent
From Lemma \ref{lemma-1}, it is easy to get
\begin{equation}
\label{vv61}
|(q_n, \lambda_n^{-1}A^{\frac{\beta}{2}}\theta_{n})|\leq\|q_n\|
\|\lambda_n^{-1}A^{\frac{\beta}{2}}\theta_{n}\|={\lambda_n^{-k}}o(1), \end{equation}
and thanks to   $\alpha<\frac{\beta+\gamma}{2}$, we have
\begin{equation}
\label{vv6}
|(A^{\alpha}v_n, \lambda_n^{-1}\theta_{n})|\leq \|A^{\alpha-\frac{\beta}{2}}v_n\|\|\lambda_n^{-1}A^{\frac{\beta}{2}}\theta_{n}\|={\lambda_n^{-\frac{k}{2}}}o(1).
\end{equation}
Thus,   from \eqref{v3_0}-\eqref{vv6}, we obtain
\begin{equation}\label{v3_1}
\|\theta_{n}\|={\lambda_n}^{-\frac{k}{4}}o(1).
\end{equation}
Note that we still have (\ref{316-}) now. To show the contradiction, we divide $V_3$ into four regions (see Figure \ref{8134}), and we shall prove $\|v_n\|,$ $\|A^{\frac{\gamma}{2}}v_n\|,$ $\|A^{\frac{1}{2}}u_n\|=o(1)$ by using (\ref{316-}) in each region, respectively(Case iii-A, B, C, D).

\begin{figure}
	\centering
	\includegraphics[width=6cm]{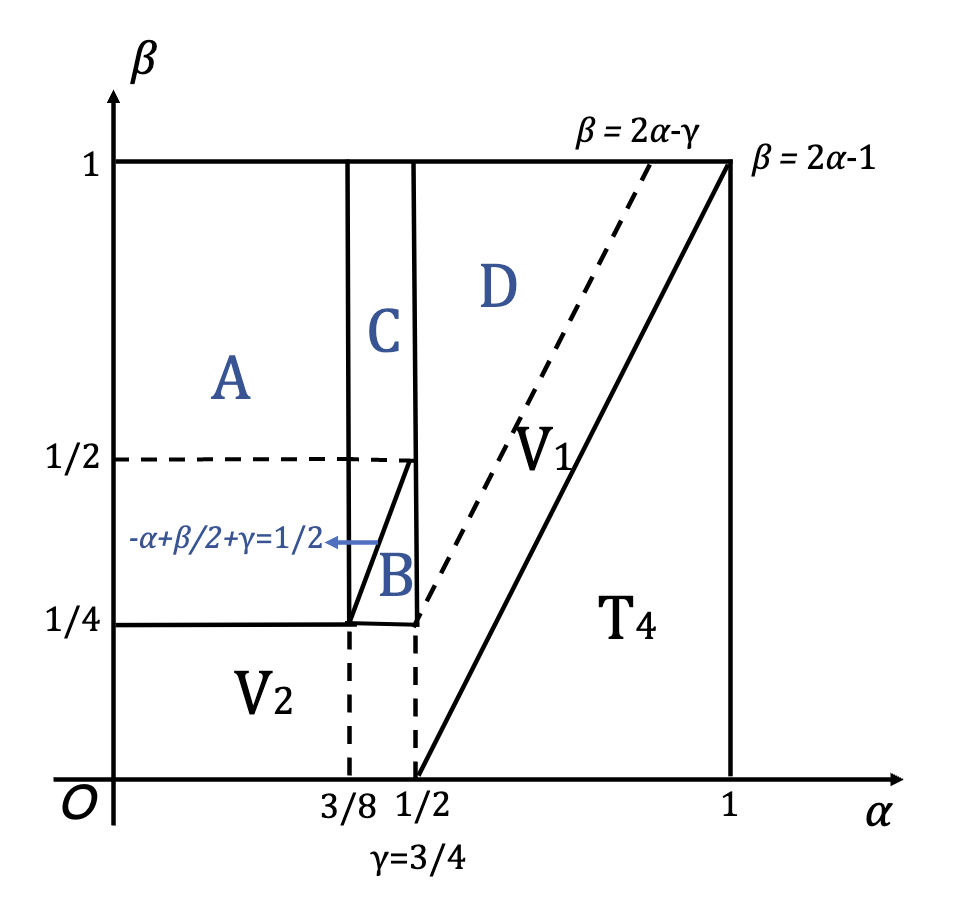}
	\caption{\small Subregions of $V_3$.}
	\label{8134}
\end{figure}

\indent
{\bf Case iii-A.} Let $(\alpha, \beta, \gamma) \in V_3$ and   $0\leq \alpha \leq \frac{\gamma}{2}$.

The assumption $\beta+\gamma\ge 1$ and \eqref{s2+}, \eqref{+211}, \eqref{v3_1} imply that
\begin{equation*}
\label{vv7}
\| A^{-\frac{\beta+\gamma}{2}}\left(I+mA^{\gamma}\right)v_n\|
\le \lambda_n^{-1}(\sigma \|A^{1-\frac{\beta+\gamma}{2}} u_n\|
+\|A^{\alpha-\frac{\beta+\gamma}{2}}\theta_n\|+ \lambda_n^{-k}o(1))
= \lambda_n^{-1}O(1).
\end{equation*}
Consequently,   we have
\begin{equation}
\label{vv8}
\| A^{{1\over2} - \beta- \gamma}\left(I+mA^{\gamma}\right)v_n\|=\| A^{{1\over2}-\frac{\beta+\gamma}{2}}A^{-\frac{\beta+\gamma}{2}}\left(I+mA^{\gamma}\right)v_n\|
= \lambda_n^{-1}O(1).
\end{equation}
\noindent
Taking the inner product of (\ref{s3+}) with $\lambda_n^{-k}A^{ \frac{1}{2}-\frac{\beta}{2}-\gamma}(I+mA^{\gamma})v_n$ {gives}
\begin{equation} \label{vv9}
\begin{array}{l}
i\lambda_n(A^{\frac{\beta}{2}}\theta_n, A^{\frac{1}{2}-\beta-\gamma}(I+mA^{\gamma})v_n) -( q_n,A^{{1\over2} - \gamma}(I+mA^{\gamma})v_n)
\\ \noalign{\medskip}  \displaystyle +\|A^{{1\over2}( \alpha-\frac{\beta}{2}-\gamma+\frac{1}{2})} v_n\|^2+m\|A^{{1\over2} ( \alpha-\frac{\beta}{2}+\frac{1}{2}) } v_n\|^2=\lambda_n^{-k}o(1).
\end{array}
\end{equation}

\noindent
We deduce from \eqref{v1_0}-\eqref{+211} and \eqref{vv8} that \begin{equation}
\label{v9}
i\lambda_n(A^{\frac{\beta}{2}}\theta_n, A^{\frac{1}{2}-\beta-\gamma}(I+mA^{\gamma})v_n),\;\;
( q_n,A^{{1\over2} - \gamma}(I+mA^{\gamma})v_n) = \lambda_n^{1-\frac{k}{2}}o(1).
\end{equation}
Combining \eqref{vv9}-\eqref{v9} yields
\begin{equation}\label{v3_3}
\|A^{{1\over2}( \alpha-\frac{\beta}{2}-\gamma+\frac{1}{2})} v_n\|,  \quad
\|A^{{1\over2} ( \alpha-\frac{\beta}{2}+\frac{1}{2}) } v_n\| =\lambda_n^{\frac{1}{2}-\frac{k}{4}}o(1).
\end{equation}

\noindent
Note that ${1\over2}(\alpha-\frac{\beta}{2}+\frac{1}{2})\le \frac{\gamma}{2}<\frac{1}{2}$ due to
$\alpha \le {\gamma\over2}$ and $\beta+\gamma\ge1. $
Then using Lemma \ref{lemma-inter},  along with \eqref{+211} and \eqref{v3_3}, one has
\begin{equation*}
\|A^{\frac{\gamma}{2}}v_n\|\leq\|A^{{1\over2}(\alpha-\frac{\beta}{2}+\frac{1}{2})} v_n\|^{1-\frac{\beta+2\gamma-2\alpha-1}{1-2\alpha+\beta}}
\|A^{\frac{1}{2}}v_n\|^{\frac{\beta+2\gamma-2\alpha-1}{1-2\alpha+\beta}} = o(1).
\end{equation*}
{Thus, \eqref{vv1} holds in Case iii-A.} It is clear that
$|(A^{\alpha}\theta_n, u_n)| \le \| \theta_n\| \|A^{\alpha} u_n\| = o(1)$ by $\alpha\le {\gamma\over2} \leq{1\over2}$, \eqref{unitnorm} and \eqref{v3_1}. Therefore, by (\ref{316-}) along with (\ref{vv1}), we obtain $\|A^{1\over2} u_n\| = o(1)$, and consequently, $\|U_n\|_{\cal H} = o(1) $ by \eqref{v1_0}, \eqref{316-}, \eqref{vv1} and \eqref{v3_1}.

{\bf Case iii-B.}
Let $(\alpha, \beta, \gamma) \in V_3,\;\frac{\gamma}{2}<\alpha\leq\frac{1}{2}$ and $
-\alpha+\frac{\beta}{2}+\gamma\le \frac{1}{2}$.

Similar to (i), one has that  \eqref{v1_5} still holds since $\alpha >{\gamma \over2}.$
Let us estimate the first two terms of \eqref{v1_5}.
First, note that $0<1-\alpha-\frac{\beta}{2}\leq\frac{1}{2}$ since $\frac{\gamma}{2}<\alpha\leq\frac{1}{2}$
and $\gamma+\beta \ge1.$
From \eqref{+211} {and}  \eqref{+211-1}, we can deduce that
\begin{equation}
\label{vv10}
\|A^{1-\alpha-\frac{\beta}{2}}u_n\| \le
\left\{\begin{array}{ll}
\lambda_n^{-1}O(1), & 0< 1-\alpha-\frac{\beta}{2}\le {\gamma\over2},
\\ \noalign{\medskip}  \displaystyle
\|A^{\frac{\gamma}{2}}u_n\|^{1-\frac{2-2\alpha-\beta-\gamma}{1-\gamma}}\|A^{\frac{1}{2}}u_n\|^{\frac{2-2\alpha-\beta-\gamma}{1-\gamma}}
=\lambda_n^{\frac{1-2\alpha-\beta}{1-\gamma}}O(1),
& {\gamma\over2} < 1\!-\!\alpha\!-\frac{\beta}{2}\le {1\over2}.
\end{array}\right.
\end{equation}
Consequently, by \eqref{v1_8} and \eqref{vv10},
{\small\begin{equation}
\label{vv11}
|(A^{{\beta\over2}}\theta_n, A^{1-\alpha-\frac{\beta}{2}}u_n)| \le \|A^{{\beta\over2}}\theta_n\| \|A^{1-\alpha-\frac{\beta}{2}}u_n\|=
\left\{\begin{array}{ll}
\lambda_n^{-{k\over2}}o(1), & 0< 1-\alpha-\frac{\beta}{2}\le {\gamma\over2},
\\ \noalign{\medskip}  \displaystyle
\lambda_n^{\frac{2(1-\beta-\gamma)}{1-\gamma}}o(1),
& {\gamma\over2} < 1-\alpha-\frac{\beta}{2}\le {1\over2}.
\end{array}\right.
\end{equation}}
Note that  $\beta+\gamma\ge 1$, {one has}
\begin{equation}
\label{vv12}
|(A^{1-\alpha}\theta_n, u_n)| = o(1).
\end{equation}

\noindent By $\frac{\gamma}{2}<-\alpha+\frac{\beta}{2}+\gamma\le \frac{1}{2}$, we have the following estimation by Lemma \ref{lemma-inter}, \eqref{+211} {and} \eqref{+211-1},
\begin{equation*}
\label{vv13}
\|A^{-\alpha+\frac{\beta}{2}+\gamma}v_n\|
\leq\|A^{\frac{\gamma}{2}}v_n\|^{1-\frac{-2\alpha+\beta+\gamma}{1-\gamma}}
\|A^{\frac{1}{2}}v_n\|^{\frac{-2\alpha+\beta+\gamma}{1-\gamma}}
=\lambda_n^{\frac{-2\alpha+\beta+\gamma}{1-\gamma}}O(1).
\end{equation*}
This combined with \eqref{v1_0} yields
\begin{equation}
\label{vv14}
|(q_n, A^{-\alpha+\frac{\beta}{2}}(I+m A^{\gamma}) v_n)| \le C \|q_n\| \| A^{-\alpha+\frac{\beta}{2}+\gamma } v_n\|= o(1).
\end{equation}

\noindent Substituting \eqref{v3_1}, \eqref{vv12}, \eqref{vv14} into \eqref{v1_5}, we get $\|v_n\|, \|A^{\gamma\over 2} v_n\| = o(1),$ i.e., \eqref{vv1} holds.
Then,  similar to the discussion in { Case iii-A}, we can get $|(A^{\alpha}\theta_n, u_n)| \le \| \theta_n\| \|A^{\alpha} u_n\| = o(1)$ due to $\alpha\le {1\over2}$ and \eqref{v3_1}, which together with (\ref{316-}) and (\ref{vv1})  implies $\|A^{1\over2} u_n\| = o(1)$.  Finally,
combining this with \eqref{v3_1}, we get the contradiction $\|U_n\|_{\cal H} = o(1).$

\indent
{\bf Case iii-C.} Let $(\alpha, \beta, \gamma) \in V_3,\;\frac{\gamma}{2}<\alpha\leq\frac{1}{2},$ and  $
-\alpha+\frac{\beta}{2}+\gamma> \frac{1}{2} $.

It follows from  \eqref{s3+} that
\begin{equation}\label{200}
\|A^{\frac{\beta}{2}}q_n\| \le \|\lambda_n\theta_n\|+ \|A^\alpha v_n\| + \lambda_n^{-k}o(1).
\end{equation}
{By \eqref{+211}, \eqref{+211-1} and interpolation \eqref{4.41},} one gets
\begin{equation}
\label{vv15}
\|A^\alpha v_n\| \le \|A^{\gamma \over2} v_n\|^{1-\frac{2\alpha-\gamma}{1-\gamma}} \|A^{1\over2} v_n\|^{\frac{2\alpha-\gamma}{1-\gamma}} = \lambda_n^{\frac{2\alpha-\gamma}{1-\gamma}}O(1).
\end{equation}

\noindent
We shall show that
\begin{equation}\label{3.31}
\|\lambda_n\theta_n\|=\left\{
\begin{array}{ll} \displaystyle
\lambda_n^{1-\frac{k}{4}}o(1),\quad &\mbox{if}~~~
1-\frac{k}{4}\leq{\frac{2\alpha-\gamma}{1-\gamma}},\\   \noalign{\medskip}  \displaystyle
\lambda_n^{\frac{2\alpha-\gamma}{1-\gamma}}o(1), \quad &\mbox{if}~~~  1-\frac{k}{4}>{\frac{2\alpha-\gamma}{1-\gamma}}.
\end{array}\right.
\end{equation}

\noindent
Therefore, combining \eqref{200}-\eqref{3.31}, we see
\begin{equation}\label{3.32}
\|A^{\frac{\beta}{2}}q_n\|=\lambda_n^{\frac{2\alpha-\gamma}{1-\gamma}}O(1).
\end{equation}
Thanks to
$0<-\alpha+\frac{\beta}{2}+\frac{\gamma}{2}<\frac{\beta}{2}$, Lemma \ref{lemma-inter}, \eqref{v1_0}, (\ref{+211-1}) and \eqref{3.32}, we obtain
\begin{align*}
    \begin{array}{lll}
|(q_n, A^{-\alpha+\frac{\beta}{2}}(I+m A^{\gamma}) v_n)|&\le& C\|A^{-\alpha+\frac{\beta}{2}+\frac{\gamma}{2}}q_n\|\|A^\frac{\gamma}{2}v_n\|\\ \noalign{\medskip}
&\le&
C\|A^{-\alpha+\frac{\beta}{2}+\frac{\gamma}{2}}q_n\|  \\ \noalign{\medskip} \displaystyle
& \le& C\|q_n\|^{\frac{2\alpha-\gamma}{\beta}}\|A^{\frac{\beta}{2}}q_n\|^{1-\frac{2\alpha-\gamma}{\beta}}\\
&=& o(1).
\end{array}
\end{align*}
Then, using the same argument after \eqref{vv14}, one can also arrive at the contradiction $\|U_n\|_{\cal H} = o(1).$

{Now we are the place to prove \eqref{3.31}.}
By \eqref{v3_1}, it is obvious that
\begin{equation}
\label{vv17}
\|\lambda_n \theta_{n}\|=\lambda_n^{1-\frac{k}{4}}o(1) =  \lambda_n^{\frac{2+2\alpha-\beta-3\gamma}{ 2(1-\gamma)} } o(1).
\end{equation}
On the other hand,  if $1-\frac{k}{4}>{\frac{2\alpha-\gamma}{1-\gamma}}$, i.e., $\alpha+\frac{\beta+\gamma}{2}<1$, one can get a better decay rate than \eqref{vv17}.
In fact, recall that  $\frac{1}{2}-\beta-\frac{\gamma}{2}\leq 0$ due to $\beta+\gamma \ge 1 $.
Then, it follows from   (\ref{s2+}) that
\begin{equation}
\label{vv18}
i\lambda_nA^{\frac{1}{2}-\beta}(mI+A^{-\gamma})v_n+\sigma A^{\frac{3}{2}-\beta-\gamma}u_n-A^{\frac{1}{2}+\alpha-\beta-\gamma}\theta_{n}=o(1).
\end{equation}

\noindent
By \eqref{+211}, \eqref{v3_1}, \eqref{vv18},  $\frac{3}{2}-\beta-\gamma\leq \frac{1}{2} $ and $\frac{1}{2}+\alpha-\beta-\gamma\leq 0$, one has
\begin{equation}\label{3.3}
\|A^{\frac{1}{2}-\beta}v_n\|=\lambda_n^{-1}O(1).
\end{equation}
{Note that  $\frac{1}{2}-\beta\le\alpha-\frac{\beta}{2}
	<-\beta-\frac{\gamma}{2}+1$ where we use $\alpha+\frac{\beta+\gamma}{2}<1$ and $\beta+\gamma\ge 1$.
	Combining \eqref{+211-1},} \eqref{3.3} and interpolation, we have
\begin{equation}
\label{vv19}
\|A^{\alpha-\frac{\beta}{2}}v_n\|\leq
\|A^{\frac{1}{2}-\beta}v_n
\|^{\frac{2-2\alpha-\beta-\gamma}{1-\gamma}}
\|A^{-\beta-\frac{\gamma}{2}+1}v_n
\|^{1-\frac{2-2\alpha-\beta-\gamma}{1-\gamma}}
=\lambda_n^{-\frac{2-2\alpha-\beta-\gamma}{1-\gamma}}O(1),
\end{equation}
where {$\|A^{-\beta-\frac{\gamma}{2}+1}v_n\|=O(1)$ } since $-\beta-\frac{\gamma}{2}+1\le\frac{\gamma}{2}$.
\noindent
Using \eqref{v1_8} and \eqref{vv19} yields
\begin{equation}
\label{vv20}
|\lambda_n^{-1}(A^{\alpha}v_n, \theta_{n})|
\leq\|A^{\alpha-\frac{\beta}{2}}v_n\|
\|\lambda_n^{-1}A^{\frac{\beta}{2}}\theta_{n}\|
=\lambda_n^{\frac{4\alpha-2}{1-\gamma}}o(1).
\end{equation}
Note that $-k\leq \frac{4\alpha-2}{1-\gamma}$ due to $\beta+\gamma\ge 1$. Thus,
substituting \eqref{vv61} and \eqref{vv20} into   (\ref{v3_0}) yields $\|\theta_{n}\|=\lambda_n^{\frac{2\alpha-1}{1-\gamma}}o(1)$. Consequently, the second part of \eqref{3.31} is obtained.

\indent
{\bf Case iii-D.}  Let   $\frac{1}{2}<\alpha\leq\frac{\beta+\gamma}{2}$.

Note  that by (\ref{s2+}), (\ref{+211}) and (\ref{v3_1}), along with  $\alpha>\frac{1}{2}$, we see that
\begin{equation}\label{111}
\|\lambda_nA^{-\alpha}(I+mA^{\gamma})v_n\|=O(1).
\end{equation}
Thus, taking the inner product of \eqref{s3+} with
$\lambda_n^{-k}A^{-\alpha}(I+mA^{\gamma})v_n$ in $H$, we have
\begin{equation}\label{v3_5}
(i\lambda_n\theta_n, A^{-\alpha}(I+mA^{\gamma})v_n)-(A^{\frac{\beta+\gamma}{2}-\alpha}q_n, A^{-\frac{\gamma}{2}}(I+mA^{\gamma})v_n)+\|v_n\|^2+m\|A^{\frac{\gamma}{2}}v_n\|^2 = o(1).
\end{equation}

\noindent
Recalling that $\alpha>\frac{1}{2}$ and \eqref{s3+}, one has
$$\|i\lambda_nA^{\frac{1}{2}-\alpha}\theta_n-A^{\frac{1}{2}-\alpha+\frac{\beta}{2}}q_n+A^{\frac{1}{2}} v_n\|=\lambda_n^{-k}o(1).$$
Thus, we deduce from \eqref {+211} and \eqref{v3_1} that
\begin{equation}
\label{vv21}
\|A^{\frac{\beta+1-2\alpha}{2}}q_n\|=\lambda_nO(1).
\end{equation}

\noindent
By interpolation \eqref{4.41}, along with \eqref{v1_0} and \eqref{vv21}, we get
\begin{equation}\label{vv22} \|A^{\frac{\beta+\gamma}{2}-\alpha}q_n\|\leq\|q_n\|^{1-{\frac{\beta+\gamma-2\alpha}{\beta+1-2\alpha}}}\|A^{\frac{\beta+1-2\alpha}{2}}q_n\|^{\frac{\beta+\gamma-2\alpha}{\beta+1-2\alpha}}=o(1).
\end{equation}

\noindent
Note that $\|A^{-\frac{\gamma}{2}}(I+mA^{\gamma})v_n\|\leq \| A^{-\frac{\gamma}{2}}v_n\|+
\|mA^{\frac{\gamma}{2}}v_n\|=O(1)$. Substituting \eqref{v3_1}, (\ref{111}) and \eqref{vv22} into \eqref{v3_5} yields
\begin{equation}\label{222}
\|v_n\|,\;\|A^{\frac{\gamma}{2}}v_n\| = o(1).
\end{equation}

\noindent Finally, the assumption $0<\alpha-\frac{\beta}{2}<\frac{\gamma}{2}$, \eqref{s1+} and \eqref{v1_8} imply that
\begin{equation}
\label{v23}
|(A^{\alpha}\theta_n, u_n)|\leq \| \lambda_n^{-1}A^{\frac{\beta}{2}}\theta_n\| \| \lambda_nA^{\alpha-\frac{\beta}{2}}u_n\|\leq\|\lambda_n^{-1}A^{\frac{\beta}{2}}\theta_n\|\|A^{\alpha-\frac{\beta}{2}}v_n\|=\lambda_n^{-\frac{k}{2}}o(1).
\end{equation}
Then, by \eqref{316-}, along with (\ref{222}) and (\ref{v23}), we obtain
\begin{equation}\label{333}
\|A^{\frac{1}{2}}u_n\|=o(1).
\end{equation}
In summary,  by (\ref{v1_0}), (\ref{v3_1}), (\ref{222}) and (\ref{333}), we have arrived at the contradiction $\|U_n\|_{\mathcal{H}}=o(1)$.

\section{Polynomial Stability of the System without an Inertial Term (Proof of \Cref{m00})}\label{8.7}
\setcounter{equation}{0}
This section is devoted to considering the stability of system \eqref{101} without an inertial term, i.e., $m=0$. Note that by Lemma \ref{l202}, we know that the corresponding semigroup $e^{t{\cal A}}$ is strongly stable. We shall further estimate the polynomial decay rates of the solutions to the system when parameters $(\alpha,\;\beta)\in V_i^*,\; i=1,2,$ respectively.

Similar to the argument in Section \ref{8.8},  we still employ the proof by contradiction to show \Cref{m00}. Specifically,  suppose (\ref{huang2}) fails. Then, there at least exists  a sequence $\{ \lambda_{n},  U_{n}  \}_{n=1}^\infty
\subset {\mathbb R} \times \mathcal{D}({\cal A})  $ such
that \eqref{unitnorm}-\eqref{331} hold with $m=0$, that is
\begin{eqnarray}
&&\lambda_n^k(i\lambda_nA^{\frac{1}{2}}u_n- A^{\frac{1}{2}}v_n)=o(1)\quad  \;  \mbox{ in  }  \;  H,\label{s1+1}\\   \noalign{\medskip}  \displaystyle
&&\lambda_n^k \left(i\lambda_nv_n +\sigma A u_n-A^{\alpha}\theta_n\right)=o(1)\quad  \;  \mbox{ in  }  \;  H,\label{s2+1}\\\noalign{\medskip}  \displaystyle
&&\lambda_n^k(i\lambda_n\theta_n-A^{\frac{\beta}{2}}q_n+A^\alpha v_n)=o(1)\quad  \;  \mbox{ in  }  \;  H,\label{s3+1}\\\noalign{\medskip}  \displaystyle
&&\lambda_n^k(i\lambda_n\tau  q_n+ q_n+A^{\frac{\beta}{2}}\theta_n)=o(1)\quad  \;  \mbox{ in  }  \;  H.\label{s4+1}
\end{eqnarray}
Consequently, \eqref{v1_0}-\eqref{+211} remain true.
Now, we proceed to show (i)-(ii) in \Cref{m00}, respectively.

\textbf{(i)} Let $(\alpha,\;\beta)\in V_1^*$(see (\ref{par2})) and $k= \frac{2\alpha-\beta}{\alpha}$.

\indent
Thanks to $-\alpha+\frac{\beta}{2}\le 0$ and \eqref{unitnorm}, we can deduce from  \eqref{s3+1}   that
\begin{equation}
\label{mv2}
\|\lambda_nA^{-\alpha}\theta_n\|=O(1).
\end{equation}
Then, using (\ref{v1_8}), (\ref{mv2}) and the interpolation inequality (\ref{4.41}) in Lemma \ref{lemma-inter}, we obtain
\begin{equation}\label{m05}
\left\| \theta_n\right\| \leq \left\| A^{-\alpha} \theta_n\right\| ^{\frac{\beta}{2\alpha+\beta}}\|A^{\frac{\beta}{2}} \theta_n\|^{\frac{2\alpha}{2\alpha+\beta}}=o(1).
\end{equation}

\noindent  We take  the inner product of (\ref{s3+1}) with $\lambda_n^{-k}A^{-\alpha}v_{n}$ to get
\begin{equation}\label{m06}
(i\lambda_n\theta_n, A^{-\alpha}v_n)-(A^{\frac{\beta}{2}-\alpha}q_n, v_n)+\|v_n\|^2 =o(1).
\end{equation}

\noindent
By (\ref{s2+1}), it is clear that $\|i \lambda_nA^{-\alpha} v_n +\sigma A^{1-\alpha}u_n - \theta_n \|=\lambda_n^{-k}o(1)$.  Combining this with \eqref{unitnorm}, along with $\alpha\geq \frac{1}{2}$, yields
\begin{equation}
\label{mv1}
\|A^{-\alpha}v_n\|=\lambda_n^{-1}O(1).
\end{equation}

\noindent
Substituting \eqref{unitnorm}, \eqref{v1_0}, \eqref{m05}, \eqref{mv1} into \eqref{m06},  along with $\frac{\beta}{2}-\alpha\leq 0$,  we  obtain
\begin{equation}\label{m07}
\|v_n\|= o(1).
\end{equation}

\noindent
Furthermore, similar to (\ref{v1_6})-\eqref{316-}, we can deduce from \eqref{s1+1} and \eqref{s2+1} that
\begin{equation}\label{m09}
-\|v_n\|^2+\sigma\|A^{\frac{1}{2}} u_n\|^2
-i(\lambda_n^{-1}A^{\frac{\beta}{2}}\theta_n, A^{\alpha-\frac{\beta}{2}}v_n) =o(1).
\end{equation}

\noindent
It follows from \eqref{v1_0}, \eqref{v1_8}, \eqref{s3+1}  and \eqref{m05} that
\begin{equation*} \label{mv3}
(\lambda_n^{-1}A^{\frac{\beta}{2}}\theta_n, \, A^{\alpha-\frac{\beta}{2}}v_n) =	(\lambda_n^{-1}A^{\frac{\beta}{2}}\theta_n, \, q_n-i\lambda_nA^{-\frac{\beta}{2}}\theta_n) + \lambda_n^{-1-k}o(1)
= o(1).
\end{equation*}
Combining this with \eqref{m07} and \eqref{m09} yields
\begin{equation}\label{m010}
\|A^{\frac{1}{2}}u_n\|=o(1).
\end{equation}
Consequently, we arrive at the contradiction $\|U_n\|_{\mathcal{H}}=o(1)$ by \eqref{v1_0}, \eqref{m05}, \eqref{m07} and   \eqref{m010}.\\

\textbf{ (ii)}
Let
$(\alpha,\;\beta)\in V_2^*$ and $k= 2(2-2\alpha-\beta)$.

Taking the inner product of  \eqref{s3+1} with $\lambda_n^{-k-1}\theta_n$, one has
\begin{equation}\label{m011}
(i\lambda_n\theta_n, \lambda_n^{-1}\theta_{n})-(A^{\frac{\beta}{2}}q_n, \lambda_n^{-1}\theta_{n})+(A^{\alpha}v_n, \lambda_n^{-1}\theta_{n})=\lambda_n^{-1-k}o(1).
\end{equation}

\noindent
By \eqref{v1_0} and \eqref{v1_8},  we get
\begin{equation}
\label{mv4}
|(A^{\frac{\beta}{2}}q_n, \lambda_n^{-1}\theta_n)| \leq\|q_n\|\|\lambda_n^{-1}A^{\frac{\beta}{2}}\theta_n\|=\lambda^{-k}o(1).
\end{equation}
Moreover, by \eqref{s2+1} and \eqref{+211}, we have
\begin{align}\label{4.40}
\|\lambda_n A^{-\frac{1}{2}}v_n\|=O(1).
\end{align}
Recalling  $0\leq\alpha<{1\over2}$, we deduce from  \eqref{4.41}, \eqref{+211} and \eqref{4.40}  that
\begin{equation*}
\label{mv5}
\|A^{\alpha-\frac{\beta}{2}}v_n\|\leq
\|A^{-\frac{1}{2}}v_n\|^{{1\over2}-\alpha+{\beta\over 2}}
\|A^{\frac{1}{2}}v_n\|^{{1\over2}+\alpha-{\beta\over 2}}
=\lambda_n^{2\alpha-\beta}O(1).
\end{equation*}

\noindent
This implies that
\begin{equation}
\label{mv7}
|(A^{\alpha}v_n, \lambda_n^{-1}\theta_{n})|\leq \|A^{\alpha-\frac{\beta}{2}}v_n\|
\|\lambda_n^{-1}A^{\frac{\beta}{2}}\theta_{n}\|
=\lambda^{4\alpha-2}o(1).
\end{equation}

\noindent
Thus, substituting  \eqref{mv4} and \eqref{mv7} into \eqref{m011}, along with $-k\leq 4\alpha-2$, yields
\begin{equation}\label{m015}
\|\theta_{n}\|=\lambda^{2\alpha-1}o(1).
\end{equation}

\noindent We claim that \begin{equation}\label{m013}
\|v_n\| = o(1).
\end{equation}
Indeed,  note that (\ref{m06}) still holds for this case. Then, in order to show (\ref{m013}), it suffices to prove that the first and second terms in (\ref{m06}) are both $o(1)$.

From \eqref{unitnorm} and \eqref{4.40},   it is easy to get
\begin{equation}
\label{mvv1} \|A^{-\alpha}v_n\|\leq\|A^{-\frac{1}{2}}v_n\|^{2\alpha}\|v_n\|^{1-2\alpha}=\lambda_n^{-2\alpha}O(1).
\end{equation}
Combining this with \eqref{m015} gives that the first term in (\ref{m06}) satisfies
\begin{equation}\label{mvv11}
|(i\lambda_n\theta_n, A^{-\alpha}v_n)|\leq\|\theta_{n}\|\|\lambda_nA^{-\alpha}v_n\|=o(1).
\end{equation}

\noindent We shall
show the second term in \eqref{m06} is also $o(1)$. In fact, if $0\leq\alpha\leq\frac{\beta}{2}$,  by interpolation inequality (\ref{4.41}), along with \eqref{unitnorm} and \eqref{s1+1}, we have
\begin{equation}\label{m012} \|A^{\frac{\beta}{2}-\alpha}v_n\|\leq\|A^{\frac{1}{2}}v_n\|^{\beta-2\alpha}\|v_n\|^{1-\beta+2\alpha}=\lambda_n^{\beta-2\alpha}O(1).
\end{equation}
Thus, by (\ref{v1_0}) and (\ref{m012}), we get
\begin{equation}
\label{mvv2}
|(q_n, A^{\frac{\beta}{2}-\alpha}v_n )|
\leq\|q_n\|\|A^{\frac{\beta}{2}-\alpha}v_n\|=\lambda_n^{2\beta-2}o(1) =o(1).
\end{equation}
If  $\frac{\beta}{2}<\alpha  < \frac{1}{2}$, 	
it is clear from \eqref{unitnorm} and \eqref{v1_0} that
\begin{equation}\label{444}
|(q_n, A^{\frac{\beta}{2}-\alpha}v_n )| \le
\|q_n\| \|A^{\frac{\beta}{2}-\alpha}v_n\|= \lambda_n^{-{k\over2}}o(1).
\end{equation}
By  (\ref{mvv2}) and (\ref{444}), we have proved that the second term in (\ref{m06}) is $o(1)$, and hence, (\ref{m013}) holds.

\indent
Moreover, it is clear that  $|(\theta_n, A^{\alpha}u_n)| \le \|\theta_n\|\|A^{\alpha}u_n\|= \lambda_n^{2\alpha-1}o(1) $ by \eqref{unitnorm}, \eqref{m015} and $\alpha < {1\over2}$. This along with  \eqref{m09} and \eqref{m013}, yields
\begin{equation}\label{555}
\|A^{\frac{1}{2}} u_n\|=o(1).
\end{equation}
In summary, by (\ref{v1_0}), (\ref{m015}), (\ref{m013}) and (\ref{555}), we have arrived at the contradiction $\|U_n\|_{\mathcal{H}}=o(1)$. The desired result follows.

\section{Proof of  Optimality of Decay Rates (\Cref{th-o})}\label{8.6}
\setcounter{equation}{0}
\setcounter{theorem}{0}

In this section, we shall prove \Cref{th-o}, which shows the orders of polynomial decay achieved in \Cref{t-3-1} and \Cref{m00} are indeed optimal. To this end, we analyze the eigenvalues $\lambda$ of the operator $\mathcal{A}$, both when $m \ne 0$ (in the presence of an inertial term) and when $m = 0$ (in the absence of an inertial term). In addition, we assume that the system has different wave speeds, that is, $\sigma \tau \ne 1$. In what follows, we first give the characteristic equation associated with $\mathcal{A}$ (\Cref{sec:characteristic-equation}). We then describe the solutions to the characteristic equation in an asymptotic setting in \Cref{sec:with-inertial} when the system is with an inertial term (\Cref{tab:sigma-2}) and without an inertial term (\Cref{tab:m-0-sigma-2}), respectively. Finally, we show that these eigenvalues indicate the optimality of the polynomial decay rates described in \Cref{t-3-1} and \Cref{m00} (\Cref{sec:optimality-remark}).

\subsection{Characteristic Equations}
\label{sec:characteristic-equation}

We shall obtain  the characteristic equations associated with the operator $\mathcal{A}$ when the system is with an inertial term ($m \ne 0$) and without an inertial term ($m=0$). For this purpose, recall that $A$ is a  self-adjoint, positive-definite operator with compact resolvent. Thus, there exists a sequence of eigenvalues $\{\m_n\}_{n\ge1}$  of $A$ such that
\begin{equation*}
0<\m_1\le\m_2\le\cdots\le\m_n\le\cdots,~~\qquad\lim_{n\to\infty}\m_n=\infty.
\end{equation*}
A direct computation gives that the eigenvalues $\lambda$ of operator ${\cal A}$ {satisfy} the following quartic equation:
\begin{equation}
\label{f=0}
\begin{array}{lcl}
f(\l,\m_n) &:=& (m\tau \cdot \mu_n^\gamma + \tau ) \lambda^4 + (m \cdot  \mu_n^\gamma +1) \lambda^3 + (\tau \cdot \mu_n^{2\alpha} + m \cdot \mu_n^{\beta+\gamma} + \sigma \tau \cdot \mu_n +  \mu_n^\beta) \lambda^2\crr
 &&+ (\mu_n^{2\alpha}+\sigma \cdot \mu_n) \lambda + \sigma \cdot \mu_n^{1+\beta}\crr
&=&0.
\end{array}
\end{equation}
{Since} the system has different wave speeds, we take $\sigma=2$ and $\tau=1$ so that $\sigma \tau = 2 \ne 1$ without loss of generality. In addition, by taking $m=1$ and $m=0$, we can derive the characteristic equation of the system from (\ref{f=0}) when it is with and without an inertial term respectively as follows:
\begin{eqnarray}
\label{eq:characteristic-with-inertial}
&(\mu_n^\gamma + 1 ) \lambda_n^4 + ( \mu_n^\gamma +1) \lambda_n^3 + (  \mu_n^{2\alpha} + \mu_n^{\beta+\gamma} + 2 \mu_n +  \mu_n^\beta) \lambda_n^2 + (\mu_n^{2\alpha}+2 \mu_n) \lambda_n + 2 \mu_n^{1+\beta}=0.
\\ \noalign{\medskip}  \displaystyle
\label{eq:characteristic-without-inertial}
&\lambda_n^4 + \lambda_n^3 + ( \mu_n^{2\alpha} +  2 \mu_n +  \mu_n^\beta) \lambda_n^2 + (\mu_n^{2\alpha}+ 2 \mu_n) \lambda_n + 2 \mu_n^{1+\beta}=0.
\end{eqnarray}
In (\ref{eq:characteristic-with-inertial}) and (\ref{eq:characteristic-without-inertial}), $\lambda_n$ are sequences of eigenvalues of $\mathcal{A}$. 
We seek to solve \eqref{eq:characteristic-with-inertial} and \eqref{eq:characteristic-without-inertial} for $\lambda_n$ in an asymptotic setting, that is, in the setting as $\mu_n \rightarrow \infty$. Several existing works \cite{kuang, han,hao2023stability,haoliu2} have described the applications of a procedure that can also be used to solve (\ref{eq:characteristic-with-inertial}) and (\ref{eq:characteristic-without-inertial}). In this paper, we apply this procedure to solve (\ref{eq:characteristic-with-inertial}) and (\ref{eq:characteristic-without-inertial}), and omit the discussion of its details. We refer readers to the aforementioned works for details of this procedure. For each region in Defintion~\ref{def:E-partition} and Defintion~\ref{def:E-tilde-partition}, the procedure applies the quartic formula \cite{irving2004integers} to identify the roots of \eqref{eq:characteristic-with-inertial} and \eqref{eq:characteristic-without-inertial}. The solutions to (\ref{eq:characteristic-with-inertial}) and (\ref{eq:characteristic-without-inertial}) are summarized in {Tables} \ref{tab:sigma-2} and \ref{tab:m-0-sigma-2}, respectively.

\subsection{Eigenvalues of the System}
\label{sec:with-inertial}

First,  we consider the solutions to (\ref{eq:characteristic-with-inertial}), which are eigenvalues of $\mathcal{A}$ with an inertial term, in each {region given in Definition \ref{def:E-partition}.} It should be noticed that the partition in {Definition} \ref{def:E-partition} is similar to that presented in \cite{han}, where the exponential stability was  established when parameters $(\alpha,\beta,\gamma) \in F_{13},\;L_{123},\; L_{34}$. More precisely, the solutions to \eqref{eq:characteristic-with-inertial}
are {presented} in \Cref{tab:sigma-2}. As a result, one can {conclude that in region $V_1$,} even though the coefficients for the real and imaginary parts of the eigenvalues on $T_1, F_{12}, F_{14},$ and $ L_{124}$ are different, there always exists a sequence of eigenvalues on each region such that $\Re \lambda_{n} = | \Im   \lambda_{n} |^{-{1\over k_1}},$ where $k_1$ is defined in \Cref{t-3-1}. Therefore, the decay properties of the corresponding semigroup are {consistent in each part of $V_1$ by \cite{l-rao}.}

By a similar argument on $V_2,\; V_3$ and $V_1^*,\; V_2^*$ respectively, we can obtain that the long-time behavior of system \eqref{101} remains the same in every part of the interior of $V_i,i=2,3$ or $V_j^*,j=1,2$. Note that to characterize the long-time behavior in $V_1^*$ and $V_2^*$, we have used the solutions to \eqref{eq:characteristic-without-inertial} provided in \Cref{tab:m-0-sigma-2}.

\begin{table}[th]
	\centering
	\bgroup
	\def\arraystretch{1.5}
	\resizebox{\textwidth}{!}{
		\begin{tabular}{crr}
			\hline
			Region &
			$\lambda_{n,1}$ and $\lambda_{n,2}$
			&
			$\lambda_{n,3}$ and $\lambda_{n,4}$\\
			\hline
			$T_{1}$ &
			$-\frac{1}{2}\m_n^{-2\a+\beta+\g}(1+o(1))\pm i\m_n^{\a-\frac{\g}{2}}(1+o(1))$ &
			$-\frac{1}{2}(1+o(1)) \pm i\sqrt{2}\m_n^{-\a+\frac{\beta}{2}+\frac{1}{2}}(1+o(1))$ \\
			$T_{2}$ &
			$-\frac{1}{8}\m_n^{2\a+\beta+\g-2}(1+o(1))\pm i\sqrt{2}\m_n^{\frac{1}{2}-\frac{\g}{2}}(1+o(1))$ & $-\frac{1}{2}(1+o(1)) \pm i\m_n^{\frac{\beta}{2}}(1+o(1))$\\
			$T_{3}$ &
			$-\frac{1}{2}(1+o(1)) \pm i\m_n^{\frac{\beta}{2}}(1+o(1))$ &
			$-\frac{1}{2}\m_n^{2\a-\beta-\g}(1+o(1)) \pm i\sqrt{2}\m_n^{\frac{1}{2}-\frac{\g}{2}}(1+o(1))$\\
			$T_{4}$ &
			$-\frac{1}{2}\m_n^{-2\a+\beta+\g}(1+o(1)) \pm i\m_n^{\a-\frac{\g}{2}}(1+o(1))$ &
			$-2\m_n^{-2\a+\beta+1}(1+o(1))$,
			$-1(1+o(1))$\\
			$F_{12}$ &
			$-\frac{1}{18}\m_n^{\beta+\g-1}(1+o(1)) \pm i\sqrt{3}\m_n^{\frac{1}{2}-\frac{\g}{2}}(1+o(1))$ &
			$-\frac{1}{2}(1+o(1)) \pm i\frac{\sqrt{6}}{3}\m_n^{\frac{\beta}{2}}(1+o(1))$ \\
			$F_{13}$ &
			$-\frac{1}{4}(1+o(1)) \pm i\sqrt{2}\m_n^{\beta/2}(1+o(1))$  &
			$-\frac{1}{4}(1+o(1)) \pm i\m_n^{\frac{1}{2}-\frac{\g}{2}}(1+o(1))$ \\
			$F_{14}$ &
			$-\frac{1}{2}\m_n^{\g-1}(1+o(1)) \pm i\m_n^{\frac{\beta}{2}-\frac{\g}{2}+\frac{1}{2}}(1+o(1))$ &
			$-\frac{1}{2}(1+o(1))\pm \frac{i\sqrt{7}}{2}(1+o(1))$\\
			$F_{2}$ &
			$-\frac{1}{8}\m_n^{2\a+\g-2}(1+o(1)) \pm i\sqrt{2}\m_n^{\frac{1}{2}-\frac{\g}{2}}(1+o(1))$ &
			$-\frac{1}{2}(1+o(1)) \pm \frac{i\sqrt{3}}{2}(1+o(1))$\\
			$F_{23}$ &
			$-\frac{1}{2}\m_n^{2\a-1}(1+o(1)) \pm i\sqrt{2}\m_n^{\frac{1}{2}-\frac{\g}{2}}(1+o(1))$ &
			$-\frac{1}{2}(1+o(1)) \pm i\m_n^{\frac{1}{2}-\frac{\g}{2}}(1+o(1))$
			\\
			$L_{123}$ &
			$(-\frac{1}{4}+\frac{\sqrt{2}}{8})(1+o(1)) \pm i\sqrt{2+\sqrt{2}}\m_n^{\frac{1}{2}-\frac{\g}{2}}(1+o(1))$ &
			$(-\frac{1}{4}-\frac{\sqrt{2}}{8})(1+o(1))\pm i\sqrt{2-\sqrt{2}}\m_n^{\frac{1}{2}-\frac{\g}{2}}(1+o(1))$\\
			$L_{124}$ &
			$-\frac{1}{18}\m_n^{\g-1}(1+o(1)) \pm i\sqrt{3}\m_n^{\frac{1}{2}-\frac{\g}{2}}(1+o(1))$ &
			$-\frac{1}{2}(1+o(1)) \pm \frac{i\sqrt{15}}{6}(1+o(1))$ \\
			$L_{2}$ &
			$-\frac{1}{6}\m_n^{2\a-1}(1+o(1)) \pm i\sqrt{2}(1+o(1))$ &
			$-\frac{1}{2}(1+o(1)) \pm \frac{i\sqrt{3}}{2}(1+o(1))$ \\
			$L_{34}$ &
			$-\frac{1}{4}(1+o(1)) \pm i\sqrt{2}\m_n^{\frac{\beta}{2}}(1+o(1))$ &
			$-\frac{1}{4}(1+o(1)) \pm \frac{i\sqrt{15}}{4}(1+o(1))$ \\
			$P_{234}$ & $-\frac{1}{4}(1-\sqrt{j_3}) (1+o(1))\pm i\frac{\sqrt{-j_1}+\sqrt{-j_2}}{4}(1+o(1))$  & $-\frac{1}{4}(1+\sqrt{j_3})(1+o(1))\pm i\frac{\sqrt{-j_1}-\sqrt{-j_2}}{4}(1+o(1))$\\
			\hline
		\end{tabular}
	}\egroup
\caption{Solutions to \eqref{eq:characteristic-with-inertial}  (In $P_{234}$, $j_1\approx -24.0858$, $j_2\approx -5.52312$, and $j_3 \approx 0.608892$ are the three roots of $j^3+29 j^2+115 j -81=0$.)}
	\label{tab:sigma-2}
\end{table}

\begin{table}[t]
	\centering
	\bgroup
	\def\arraystretch{1.5}
	\resizebox{\textwidth}{!}{
		\begin{tabular}{crr}
			\hline
			Region &
			$\lambda_{n,1}$ and $\lambda_{n,2}$
			&
			$\lambda_{n,3}$ and $\lambda_{n,4}$\\
			\hline
			$T_{1}^*$ &
			$-\frac{1}{2}\m_n^{\beta-2\a}(1+o(1)) \pm i\m_n^{\a}(1+o(1))$ &
			$-\frac{1}{2}(1+o(1)) \pm i\sqrt{2} \m_n^{-\a+\frac{\beta}{2}+\frac{1}{2}}(1+o(1))$ \\
			$T_{2}^*$ &
			$-\frac{1}{8}\m_n^{2\a+\beta-2}(1+o(1)) \pm i\sqrt{2} \m_n^{\frac{1}{2}}(1+o(1))$ &
			$-\frac{1}{2}(1+o(1)) \pm i\m_n^{\beta/2}(1+o(1))$ \\
			$T_{4}^*$ & $-\frac{1}{2}\m_n^{\beta-2\a}(1+o(1))\pm i\m_n^{\a}(1+o(1))$ & $-2\m_n^{-2\a+\beta+1}(1+o(1))$, $-1(1+o(1))$ \\
			$F^*_{12}$ &
			$-\frac{1}{18}\m_n^{\beta-1}(1+o(1)) \pm i\sqrt{3}\m_n^{\frac{1}{2}}(1+o(1))$ &
			$-\frac{1}{2}(1+o(1))\pm\frac{i\sqrt{6}}{3}\m_n^{\frac{\beta}{2}}(1+o(1))$ \\
			$F^*_{14}$ &
			$-\frac{1}{2}\m_n^{-1}(1+o(1))\pm i\m_n^{\frac{\beta}{2}+\frac{1}{2}}(1+o(1))$ &
			$-\frac{1}{2}(1+o(1))\pm i\frac{\sqrt{7}}{2}(1+o(1))$ \\
			$F^*_{2}$ &
			$-\frac{1}{8}\m_n^{2\a-2}(1+o(1))\pm i\m_n^{\frac{1}{2}}(1+o(1))$ &
			$-\frac{1}{2}(1+o(1))\pm i\frac{\sqrt{3}}{2}(1+o(1))$ \\
			$L^*_{124}$ &
			$-\frac{1}{18}\m_n^{-1}(1+o(1))\pm i\sqrt{3}\m_n^{\frac{1}{2}}(1+o(1))$ &
			$-\frac{1}{2}(1+o(1))\pm i\frac{\sqrt{15}}{6}(1+o(1))$\\
			$L^*_{23}$ &
			$-\frac{1}{2} \mu^{2\alpha-1}(1+o(1))\pm i\sqrt{2}\m_n^{\frac{1}{2}}(1+o(1))$ &
			$-\frac{1}{2}(1+o(1)) \pm i\m_n^{\frac{1}{2}}(1+o(1))$\\
			$P^*_{123}$ &
			$\left(-\frac{1}{4} + \frac{1}{4\sqrt{2}}\right) (1+o(1)) \pm i \sqrt{2+\sqrt{2}}\mu^{\frac{1}{2}} (1+o(1))$ &
			$\left(-\frac{1}{4} - \frac{1}{4\sqrt{2}}\right) (1+o(1)) \pm i \sqrt{2-\sqrt{2}}\mu^{\frac{1}{2}} (1+o(1))$\\
			\hline
		\end{tabular}
	} \egroup
\caption{Solutions to \eqref{eq:characteristic-without-inertial}}
	\label{tab:m-0-sigma-2}
\end{table}

\subsection{Optimality of the Polynomial Decay Rates}
\label{sec:optimality-remark}

With the eigenvalues of $\mathcal{A}$ as given in \Cref{tab:sigma-2} and \Cref{tab:m-0-sigma-2},  we can show that the polynomial decay rates given in \Cref{t-3-1} and \Cref{m00} are optimal.

We first show the rates achieved in
\Cref{t-3-1} are optimal. From \Cref{tab:sigma-2}, we can deduce that there exists a sequence of eigenvalues satisfying
\begin{equation}
\label{o1}
\Re \lambda_{n} = | \Im   \lambda_{n} |^{-{1\over k_i}}, \quad
\;  \mbox{ when  }  \; (\alpha, \; \beta,\; \gamma) \in V_i,\; i=1,2,3,
\end{equation}
where $k_i$ and $V_i$ are defined in \Cref{t-3-1}.
Then by the same argument as given in \cite[Corollary 4.7]{haoliu2}, we can obtain
\begin{equation}
\label{o2}
\varlimsup\limits_{\lambda\in {\mathbb R}, |\lambda|\to \infty}
|\lambda |^{-{1\over k_i}}\| (i\lambda I - {\cal A})^{-1}\|_{{\cal L}({\cal H})}\ge C>0,\quad
\;  \mbox{ when  }  \; (\alpha, \; \beta,\; \gamma) \in V_i,\; i=1,2,3.
\end{equation}
Therefore, the obtained decay rates in \Cref{t-3-1} are optimal.

Following {a} similar discussion as given above, we can also derive that the obtained decay rates given in \Cref{m00} are also optimal based on the expression of eigenvalues in \Cref{tab:m-0-sigma-2}.

\section{Examples}\label{8.5}
\setcounter{equation}{0}

This section is devoted to applying our results to some coupled PDE systems and get the optimal decay rates. 

\noindent{\bf Example 1.}
  Assume $\Omega$ is
a bounded open subset of $\mathbb{R}^n$ with smooth  boundary  {$\Gamma$.}
Consider the following thermoelastic plate equation of Cattaneo's type  (thermoelastic Euler-Bernoulli plate if $m=0$ or thermoelastic Rayleigh plate if $m>0$) (\cite{Graff,rayleigh}):
\begin{equation}\label{601}
\left\{\begin{array}{ll}
u_{tt} - m \Delta u_{tt} + \sigma \Delta^2 u -   \theta=0, &  x\in\Omega, t>0,\\
\theta_{t} +   \Delta   q+ u_{t} = 0,&  x\in\Omega, t>0, \\
\tau q_t + q - \Delta   \theta =0 , &   x\in\Omega, t>0,\\
u  =\Delta u = \theta =  q  =0 , &  x\in\Gamma, t>0,\\
u(0) =  u_0, \; u_t(0)=u_1,  \; \theta(0)=\theta_0, \;  q(0)=q_0, &  x\in\Omega.
\end{array}\right.
\end{equation}

\noindent
To apply the abstract results in {Theorems }\ref{t-3-1} and \ref{m00} to this system, we let operator  $A=\Delta^2$ be the bi-Laplace operator on $\Omega  $ with domain
$ \mathcal{D}(A) = \{u\in H^4(\Omega)\,|\,
u |_{ \Gamma }=  \Delta u |_{\Gamma } =0
\},\; H= L^2(\Omega).$
Note that system (\ref{601}) is corresponding to the abstract system \eqref{101} with $(\alpha, \, \beta,\,\gamma) = (0,\;1, \;{1\over2})$  when $m>0.$ Thus,
the semigroup corresponding to \eqref{601} is polynomially stable with order $t^{-{1\over 6}}$ since parameters $(\alpha, \, \beta,\,\gamma) = (0,\;1, \;{1\over2}) \in V_3$   in \Cref{t-3-1}.

 While for the case $m=0$, the decay rate of the semigroup is $t^{-{1\over 2}}$ due to  \Cref{m00} with parameters $(\alpha, \, \beta ) = (0,\;1) \in V_2^*$.

 \vskip 4mm

\noindent{\bf Example 2.}  Let $\Omega,\Gamma$ be defined as in Example 1.  Consider the following model:
 \begin{equation}\label{602}
\left\{\begin{array}{ll}
u_{tt} - m \Delta u_{tt} + \sigma \Delta^2 u +\Delta  \theta=0, &  x\in\Omega, t>0,\\
\theta_{t}- q- \Delta u_{t} = 0,&  x\in\Omega, t>0, \\
\tau q_t + q +  \theta =0 , &   x\in\Omega, t>0,\\
u  =\Delta u = \theta =  q  =0 , &  x\in\Gamma, t>0,\\
u(0) =  u_0, \; u_t(0)=u_1,  \; \theta(0)=\theta_0, \;  q(0)=q_0, &  x\in\Omega.
\end{array}\right.
\end{equation}

\noindent
Let  $H= L^2(\Omega)$,   $A=\Delta^2$  and
$ \mathcal{D}(A) = \{u\in H^4(\Omega)\,|\,
u |_{ \Gamma }=  \Delta u |_{\Gamma } =0\}.$
By \Cref{t-3-1}, we obtain that $(\alpha, \, \beta,\,\gamma) = ({1\over2},\;0, \;{1\over2}) \in V_1$ and
the semigroup associated with system (\ref{602}) decays polynomially  with optimal  order $t^{-{1\over 2}}$.
Moreover, for the case $m=0$, the semigroup corresponding to  system \eqref{602} decays polynomially  with optimal  order $t^{-{1\over 2}}$ since $(\alpha, \, \beta ) = ({1\over2},\;0) \in V_1^*$ by applying \Cref{m00}.

\vskip 4mm

 \noindent{\bf Example 3.}   Let $\Omega,\Gamma$ be defined as in Example 1. Consider the following model:
\begin{equation}\label{604}
	\left\{\begin{array}{ll}
		u_{tt} - m \Delta u_{tt} + \sigma \Delta^2 u -   \Delta^2 \theta=0, &  x\in\Omega, t>0,\\
		\theta_{t} + \Delta q + \Delta^2 u_{t} = 0,&  x\in\Omega, t>0, \\
\tau q_t + q - \Delta \theta =0 , &   x\in\Omega, t>0,\\
	 u  =\partial_\nu u = \theta =\partial_\nu \theta= q=0 , &  x\in\Gamma, t>0,\\
    u(0) =  u_0, \; u_t(0)=u_1,  \; \theta(0)=\theta_0, \;  q(0)=q_0, &  x\in\Omega.
	\end{array}\right.
\end{equation}
Let  $H= L^2(\Omega)$,   $A=\Delta^2$  and
$ \mathcal{D}(A) = \{u\in H^4(\Omega)\,|\,
u |_{ \Gamma }=   \partial_\nu u|_{\Gamma } =0\},$ where $\nu$ is the outward
unit normal vector to the boundary. If $m>0$, using \Cref{t-3-1} yields that  $(\alpha, \; \beta,\; \gamma) = (1,\,1,\,{1\over2})\in V_1$ and
the semigroup associated with system (\ref{604}) decays polynomially  with optimal  order $t^{-\frac{3}{2}}$.
When $m=0,$  by applying \Cref{m00}, the solution   decays polynomially  with optimal  order $t^{-1}$ since $ (\alpha, \; \beta ) = (1,\,1 )\in V_1^\ast.$

\section{Conclusions}\label{8.4}

In this work, we investigated the polynomial stability of abstract thermoelastic systems with Cattaneo's law. We have achieved the following key results:

\begin{itemize}[leftmargin=*]
\item For the systems with an inertial term  ($m>0$),
we systematically classified the ``non-exponential stability''  parameters region into three distinct subregions. For each subregion, we derived explicit polynomial decay rates  that are contingent on the specific choice of system parameters
$(\alpha,\beta,\gamma)$.

\item
For the systems  without an inertial term ($m=0$), we similarly partitioned the ``non-exponential stability" region into two subregions and provided polynomial decay  rates  for  both subregions.

\item
Furthermore, we conducted a detailed asymptotic spectral analysis on the system operator in the presence or absence of an inertial term. This analysis reveals the asymptotic behavior of the eigenvalues of the system operator. Moreover, it first provides us with the candidates for the growth order of the system resolvent operator on the imaginary axis in each subregion, then allows us to confirm the optimality of the aforementioned decay rates.
\end{itemize}

Upon comparing the stability outcomes between systems with and without an inertial term, it is evident that this term plays a significant influence on the stability characteristics of these coupled partial differential equation systems.

Future work will aim to investigate the stability of system \eqref{101} when parameters satisfy $(\alpha, \beta, \gamma)\in E$ and $\beta< 2\alpha-1$.
 It is worth noting that the presence of zero as a spectrum point of the system operator in this case, as mentioned in Remark \ref{R-2-1}, poses challenges when discussing the long-term behavior of this system.
Batty \etal \cite{batty} and Rozendaal \etal \cite{Rozendaal} {provide} novel methods to investigate the stability of semigroups, which could be used to help us further identify to what extent the decay rates of the system with $\beta< 2\alpha-1$ can be achieved.
Moreover, the regularity of the semigroup associated with system \eqref{101} is another interesting issue, which is worth studying in the future.

\bibliographystyle{plain}
\bibliography{reference0604}

\begin{thebibliography}{10}

\bibitem{Ammar}
F.~Ammar-Khodja, A.~Bader, and A.~Benabdallah.
\newblock Dynamic stabilization of systems via decoupling techniques.
\newblock {\em ESAIM: Control, Optimisation and Calculus of Variations},
  4:577--593, 1999.

\bibitem{Arendt}
W.~Arendt and C.~J.~K. Batty.
\newblock Tauberian theorems and stability of one-parameter semigroups.
\newblock {\em Transactions of the American Mathematical Society},
  306(2):837--852, 1988.

\bibitem{Avalos}
G.~Avalos and I.~Lasiecka.
\newblock Exponential stability of a thermoelastic system with free boundary
  conditions without mechanical dissipation.
\newblock {\em SIAM Journal on Mathematical Analysis}, 29(1):155--182, 1998.

\bibitem{batty}
C.~J.~K. Batty, R.~Chill, and Y.~Tomilov.
\newblock Fine scales of decay of operator semigroups.
\newblock {\em Journal of the European Mathematical Society}, 18(4):853--929,
  2016.

\bibitem{bobtov1}
A.~Borichev and Y.~Tomilov.
\newblock Optimal polynomial decay of functions and operator semigroups.
\newblock {\em Mathematische Annalen}, 347:455--478, 2010.

\bibitem{Cattaneo}
C.~Cattaneo.
\newblock Sulla conduzione del calore.
\newblock {\em Atti del Seminario Matematico e Fisico dell'Università di
  Modena}, 3:83--101, 1948.

\bibitem{Chand}
D.~S. Chandrasekharaiah.
\newblock {Hyperbolic Thermoelasticity: A Review of Recent Literature}.
\newblock {\em Applied Mechanics Reviews}, 51(12):705--729, 1998.

\bibitem{conti}
M.~Conti, L.~Liverani, and V.~Pata.
\newblock Thermoelasticity with antidissipation.
\newblock {\em Discrete and Continuous Dynamical Systems. Series S},
  15(8):2173--2188, 2022.

\bibitem{rivera2}
F.~Dell’Oro, J.~E. Mu\~{n}oz Rivera, and V.~Pata.
\newblock Stability properties of an abstract system with applications to
  linear thermoelastic plates.
\newblock {\em Journal of Evolution Equations}, 4(13):777--794, 2013.

\bibitem{kuang}
H.~D. Fern{\'a}ndez~Sare, Z.~Kuang, and Z.~Liu.
\newblock Regularity analysis for an abstract thermoelastic system with
  inertial term.
\newblock {\em ESAIM: Control, Optimisation and Calculus of Variations},
  27:S24, 2021.

\bibitem{liure}
H.~D. Fern\'{a}ndez~Sare, Z.~Liu, and R.~Racke.
\newblock Stability of abstract thermoelastic systems with inertial terms.
\newblock {\em Journal of Differential Equations}, 267(12):7085--7134, 2019.

\bibitem{Graff}
K.~F. Graff.
\newblock {\em Wave motion in elastic solids}.
\newblock Dover Publications, 2012.

\bibitem{han}
Z.~Han, Z.~Kuang, and Q.~Zhang.
\newblock Stability analysis for abstract theomoelastic systems with
  {C}attaneo’s law and inertial terms.
\newblock {\em Mathematical Control and Related Fields}, 13:1639--1673, 2023.

\bibitem{hao2023stability}
J.~Hao, Z.~Kuang, Z.~Liu, and J.~Yong.
\newblock Stability analysis for two coupled second order evolution equations.
\newblock {\em Available at SSRN 4405839}, 2023.

\bibitem{haoliu1}
J.~Hao and Z.~Liu.
\newblock Stability of an abstract system of coupled hyperbolic and parabolic
  equations.
\newblock {\em Zeitschrift f{\"u}r angewandte Mathematik und Physik},
  64(4):1145--1159, 2013.

\bibitem{haoliu2}
J.~Hao, Z.~Liu, and J.~Yong.
\newblock Regularity analysis for an abstract system of coupled hyperbolic and
  parabolic equations.
\newblock {\em Journal of Differential Equations}, 259(9):4763--4798, 2015.

\bibitem{huang}
F.~Huang.
\newblock Strong asymptotic stability of linear dynamical systems in {B}anach
  spaces.
\newblock {\em Journal of Differential Equations}, 104(2):307--324, 1993.

\bibitem{irving2004integers}
R.~S. Irving.
\newblock {\em Integers, polynomials, and rings: a course in algebra}.
\newblock Springer, 2004.

\bibitem{Jou}
D.~Jou, J.~Casas-V\'{a}zquez, and G.~Lebon.
\newblock {\em Extended irreversible thermodynamics}.
\newblock Springer-Verlag, Berlin, 1996.

\bibitem{Kim}
J.~U. Kim.
\newblock On the energy decay of a linear thermoelastic bar and plate.
\newblock {\em SIAM Journal on Mathematical Analysis}, 23(4):889--899, 1992.

\bibitem{Lasiecka}
I.~Lasiecka and R.~Triggiani.
\newblock Two direct proofs on the analyticity of the s.c. semigroup arising in
  abstract thermo-elastic equations.
\newblock {\em Advances in Differential Equations}, 1998.

\bibitem{LiuLiu}
K.~Liu and Z.~Liu.
\newblock Exponential stability and analyticity of abstract linear
  thermoelastic systems.
\newblock {\em Zeitschrift f{\"u}r angewandte Mathematik und Physik ZAMP},
  48:885--904, 1997.

\bibitem{liurao}
Z.~Liu and B.~Rao.
\newblock Characterization of polynomial decay rate for the solution of linear
  evolution equation.
\newblock {\em Zeitschrift f{\"u}r angewandte Mathematik und Physik ZAMP},
  56:630--644, 2005.

\bibitem{liuzheng}
Z.~Liu and S.~Zheng.
\newblock {\em Semigroups associated with dissipative systems}.
\newblock Chapman \& Hall/CRC, Boca Raton, 1999.

\bibitem{l-rao}
P.~Loreti and B.~Rao.
\newblock Optimal energy decay rate for partially damped systems by spectral
  compensation.
\newblock {\em SIAM Journal on Control and Optimization}, 45(5):1612--1632,
  2006.

\bibitem{pazy}
A.~Pazy.
\newblock {\em Semigroups of linear operators and applications to partial
  differential equations}.
\newblock Springer-Verlag, New York, 1983.

\bibitem{rayleigh}
J.~W. S.~B. Rayleigh.
\newblock {\em The theory of sound}.
\newblock Dover Publications, New York, 1945.

\bibitem{rivera1}
J.~E.~M. Rivera and R.~Racke.
\newblock Large solutions and smoothing properties for nonlinear thermoelastic
  systems.
\newblock {\em Journal of Differential Equations}, 127(2):454--483, 1996.

\bibitem{Rozendaal}
J.~Rozendaal, D.~Seifert, and R.~Stahn.
\newblock Optimal rates of decay for operator semigroups on {H}ilbert spaces.
\newblock {\em Advances in Mathematics}, 346:359--388, 2019.

\bibitem{russell}
D.~L. Russell.
\newblock A general framework for the study of indirect damping mechanisms in
  elastic systems.
\newblock {\em Journal of Mathematical Analysis and Applications},
  173(2):339--358, 1993.

\bibitem{vernotte}
P.~Vernotte.
\newblock Les paradoxes de la th\'{e}orie continue de l'\'{e}quation de la
  chaleur.
\newblock {\em Comptes Rendus Hebdomadaires des S\'{e}ances de l'Acad\'{e}mie
  des Sciences}, 246:3154--3155, 1958.

\end{thebibliography}

\end{document}